\documentclass[11pt]{amsart}
\usepackage{graphicx}
\usepackage{amssymb}
\usepackage{amsmath}
\usepackage{amsthm}
\usepackage{mathtools}
\usepackage{tikz}
\usepackage{psfrag}
\usepackage{xcolor}

\newtheorem{thm}{Theorem}[section]
\newtheorem{lem}[thm]{Lemma}
\newtheorem{pro}[thm]{Proposition}

\theoremstyle{definition}
\newtheorem{defn}{Definition}[section]
\newtheorem{remk}{Remark}[section]

\newcommand{\N}{\mathbb N}
\newcommand{\M}{\mathbb M}

\newcommand{\R}{\mathbb R}

\makeatletter
\numberwithin{equation}{section}
\makeatother

\begin{document}

% author information

       % first author

       \author{Seonghak Kim}
       \address{Institute for Mathematical Sciences\\ Renmin University of China \\  Beijing 100872, PRC}
       \email{kimseo14@gmail.com}

       % second author

       \author{Baisheng Yan}

       % the address where the research was carried out
       \address{Department of Mathematics\\ Michigan State University\\ East Lansing, MI 48824, USA}
       \email{yan@math.msu.edu}

       % current address, usually not needed because it is the same as the
       % regular address

       % title

      % \title[On one-dimensional  forward-backward parabolic equations]{On  one-dimensional forward-backward parabolic  equations  with   linear convection and reaction}
\title[Two-phase forward solutions for forward-backward equations]{Two-phase forward  solutions for one-dimensional forward-backward  parabolic equations with linear convection and reaction}


\begin{abstract}
We study the existence and properties of Lipschitz continuous weak solutions to the Neumann boundary value problem for a  class of one-dimensional quasilinear forward-backward diffusion equations   with linear convection and reaction. The diffusion flux function is assumed to be of a  forward-backward type that contains two forward-diffusion  phases. We prove that, for all smooth initial data, there exists at least one weak solution whose spatial derivative stays  in the two forward   phases. Also, for all smooth initial data that have a derivative value lying in a certain phase transition range, we show that there exist infinitely many such solutions that exhibit instantaneous phase transitions between the two forward phases. Moreover, we introduce the notion of transition gauge for such forward  solutions and prove that the gauge  of all constructed two-phase forward   solutions can be arbitrarily close to a certain fixed constant.
\end{abstract}
\maketitle

\section{Introduction}

In this paper, we study the existence and some properties of solutions for a class of one-dimensional  quasilinear diffusion equations  of the form
\begin{equation}\label{p-main}
u_t=(\sigma(u_x))_x+b(x,t)u_x+c(x,t)u+f(x,t)\;\;\mbox{in $\Omega\times(0,T)$},
\end{equation}
coupled with the initial and Neumann boundary conditions
\begin{equation}\label{p-main-ib-condition}
\left\{
\begin{array}{ll}
  u=u_0 & \mbox{on $\Omega\times\{t=0\}$}, \\
  u_x=0 & \mbox{on $\partial\Omega\times(0,T)$.}
\end{array}\right.
\end{equation}
Here, $\Omega=(0,1)\subset\R$ is the spatial domain, $T>0$ is any fixed number, $u_0=u_0(x)\in C^{2+\alpha}(\bar\Omega)$, for a given $0<\alpha<1$, is an initial datum satisfying the compatibility condition
\begin{equation}\label{u_0-compatibility-condition}
u_0'=0\;\;\mbox{on $\partial\Omega$},
\end{equation}
and $u=u(x,t)$ is a solution to problem (\ref{p-main})-(\ref{p-main-ib-condition}) in question.

The diffusion flux function $\sigma=\sigma(s)$ in (\ref{p-main})  is assumed to be of a forward-backward type, including the well-known  H\"ollig type \cite{Ho,HN}, that contains two forward-diffusion  phases.
 More precisely, we assume that there exist numbers $s_2>s_1>0$ and $\Lambda\ge\lambda>0$ such that
\begin{equation}\label{flux-condition}
\left\{
\begin{array}{l}
  \sigma\in C^{1+\alpha}((-\infty,s_1]\cup[s_2,\infty)), \\
  \sigma'>0\;\;\mbox{in $(-\infty,s_1)\cup(s_2,\infty)$}, \\
  \lambda\le\sigma'\le\Lambda\;\;\mbox{in $(-\infty,s_1/2)\cup(2s_2,\infty)$}, \\
  \sigma(0)=0,\;\mbox{and}\; \sigma(s_1)>\sigma(s_2)\ge 0.
\end{array}\right.
\end{equation}
 Note that $\sigma$ may not be decreasing or even continuous on $[s_1,s_2]$ and thus $\sigma(s_1)$ and $\sigma(s_2)$ may not be necessarily  local extrema. See Figure \ref{fig1}  for a typical graph of such  flux functions $\sigma$.

 %Note that $\sigma$ may not be smooth or decreasing on $[s_1,s_2]$ and that $\sigma(s_1)$ and $\sigma(s_2)$ are not necessarily  local extrema.

\begin{figure}[ht]
\begin{center}
\begin{tikzpicture}[scale = .8]
    \draw[->] (-1,0) -- (11.3,0);
	\draw[->] (0,-1) -- (0,5.2);
 \draw[dashed] (0,1.72)--(5.6,1.72);
 \draw[dashed] (0,3.26)--(9.05,3.26);
 \draw[dashed] (9,0)--(9,3.3);
    \draw[dashed] (5.6, 0)  --  (5.6, 1.72) ;
   \draw[dashed] (0.85, 0)  --  (0.85, 1.72) ;
\draw[dashed] (1.36, 0)  --  (1.36, 2.4) ;
%\draw[dashed] (0,1.94)  --  (6.7, 1.94) ;
%\draw[dashed] (6.7,0)  --  (6.7, 1.94) ;
   \draw[dashed] (2.8, 0)  --  (2.8, 3.3) ;
   \draw[thick] (-0.2,-0.5)  --  (0,0) ;
	\draw[thick]   (0, 0) .. controls (2,5) and  (4, 3)   ..(5,2);
        %  \draw[thick]   (1, 1.94) .. controls (1.2,2.3) and  (7.2, 2.4)   ..(8.5,2.9);
          %\draw[thick]   (10.0, 4.03) .. controls (10.3,4.3) and  (10.8, 4.4)   ..(11.1,4.5);
	\draw[thick]   (5, 2) .. controls  (6, 1) and (9,3) ..(11.1, 5 );
%\draw[thick]   (0.9,1.5) .. controls  (1,1.6) and (2,2.6) ..(11,2.8);
	\draw (11.3,0) node[below] {$s$};
    \draw (-0.3,0) node[below] {{$0$}};
    \draw (9.6, 4) node[above] {$\sigma(s)$};
 \draw (0, 3.3) node[left] {$\sigma(s_1)$};
 \draw (0, 1.5) node[left] {$\sigma(s_2)$};
% \draw (0, 1.94) node[left] {$r_1$};
% \draw (0, 2.9) node[left] {$r_2$};
  \draw (0, 2.4) node[left] {$r$};
% \draw[dashed] (0,1.5)--(9.2, 1.5);
%\draw[dashed] (1.88,2.9)--(1.88,0);
\draw[dashed] (7.65,2.4)--(0,2.4);
\draw[dashed] (7.65,2.4)--(7.65, 0);
   \draw (5.6, 0) node[below] {$s_2$};
  % \draw (3.5, 2.2) node[above] {$\tilde\sigma(s)$};
   \draw (2.8, 0) node[below] {$s_1$};
   \draw (0.7, 0) node[below] {$s_1^*$};
  % \draw (1.2, 0) node[below] {$s_{r_1}^-$};
     \draw (1.5, 0) node[below] {$s_{r}^-$};
   \draw (9.1, 0) node[below] {$s_2^*$};
   %\draw (10.0, 0) node[below] {$s_3$};
 %  \draw (8.5, 0) node[below] {$s_{r_2}^+$};
   \draw (7.65, 0) node[below] {$s_{r}^+$};
% \draw[dashed] (0.9,0)--(0.9, 1.5);
    \end{tikzpicture}
%\end{center}
\caption{\emph{A typical graph of flux functions $\sigma.$   Here, $\sigma(s_1^*)=\sigma(s_2)$, $\sigma(s_2^*)=\sigma(s_1)$ and $\sigma(s_r^\pm)=r$.}
}
\label{fig1}
\end{center}
\end{figure}
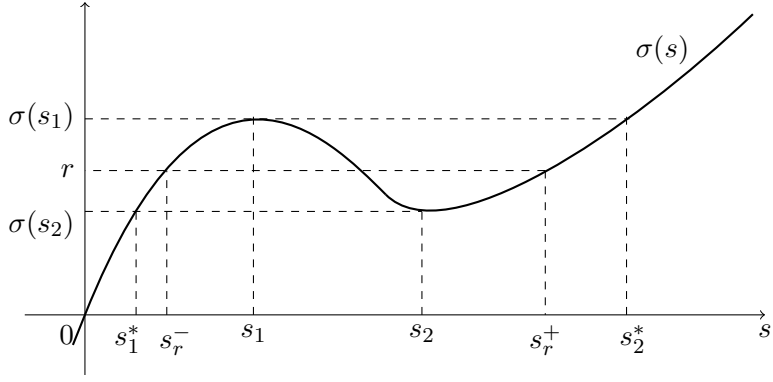

Concerning the convection coefficient  $b=b(x,t)$,   reaction coefficient $c=c(x,t)$ and the source term $f=f(x,t)$  in (\ref{p-main}), we impose the following conditions:
\begin{equation}\label{coefficient-condition}
\left\{
\begin{array}{l}
  b\in C^{\alpha,\frac{\alpha}{2}}(\bar\Omega_T)\cap C^{1,0}(\bar\Omega_T),\; c,f\in C^{\alpha,\frac{\alpha}{2}}(\bar\Omega_T),\;\mbox{and} \\
  \mbox{$c(x,t)=b_x(x,t)+d(t)$ $\forall(x,t)\in\bar\Omega_T$ for some $d\in C([0,T])$},
\end{array}\right.
\end{equation}
where $\Omega_T:=\Omega\times(0,T).$ For example, one may choose $b\in C^{1+\alpha,\frac{1+\alpha}{2}}(\bar\Omega_T)$, $d\in C^{\frac{\alpha}{2}}([0,T])$ and then take $c=b_x+d(t).$
In particular, when $b=b(t),c=c(t)\in C^{\frac{\alpha}{2}}([0,T])$, condition (\ref{coefficient-condition}) is satisfied.

Equation (\ref{p-main}) with such a function $\sigma$ of the type (\ref{flux-condition})  arises in the theory of phase transitions in thermodynamics \cite{Dy,Ho,HN,Tr}. With different types of nonmonotone flux  functions $\sigma$, equations of the type (\ref{p-main}) also stem from mathematical models of stratified turbulent flows \cite{BBDU}, population dynamics \cite{Pa}, image processing \cite{PM}, and gradient flows associated with nonconvex energy functionals \cite{BFG, Sl}.

Many  theoretical or numerical results on such forward-backward equations or related pseudoparabolic regularizations have been studied in a large number of literatures; see \cite{BBDU, BFG, Es,EG,GG1,GG2,Gu,Ho,HN,KK,KY,KY1,MTT, V,Zh, Zh1} and  references therein. For example, the existence and uniqueness of a so-called {\em two-phase entropy solution}  have been  proved in \cite{MTT} for the Cauchy problem of  equation $v_t=(\sigma(v))_{xx}$ with piecewise linear function $\sigma$ for certain {\em discontinuous} initial data $v(x,0)$. Such an equation can be related to equation $u_t=(\sigma(u_x))_x$ by  the stream function relation $v=u_x, \, (\sigma(v))_x=u_t.$ However, the notion  of  two-phase entropy solutions and related result cannot handle smooth initial data.
On the other hand, for all smooth initial data $u_0$, the existence  of Lipschitz weak solutions  was first established in \cite{Zh1} for problem (\ref{p-main})-(\ref{p-main-ib-condition}) with $\sigma$ of the type (\ref{flux-condition}) and  $b=c=f=0$ in $\Omega_T$, thereby generalizing the result of \cite{Ho}.

In this  paper, we prove that  for each smooth initial  datum $u_0$,  there exists a Lipschitz weak solution  $u$ to  problem (\ref{p-main})-(\ref{p-main-ib-condition}) that  satisfies $u_x\in (-\infty,s_1)\cup (s_2,\infty)$ a.e.\,in the space-time domain $\Omega_T$ (thus called a {\em forward solution}) and is smooth  in some part of $\Omega_T$ near its vertical boundary $\partial\Omega\times[0,T]$. Furthermore, if  $s_1^*<u_0'(x_0)<s_2^*$ for some $x_0\in\Omega$, where $s_1^*\in [0,s_1)$ and $s_2^*\in(s_2,\infty)$ are the unique  numbers with $\sigma(s_1^*)=\sigma(s_2)$ and $\sigma(s_2^*)=\sigma(s_1)$ (see Figure \ref{fig1}), then we prove the existence of  infinitely many such {forward}  solutions  $u$  to (\ref{p-main})-(\ref{p-main-ib-condition}) with  the property that
\begin{equation}\label{2-phase-condition}
\begin{cases}
\big|\{(x,t)\in \Omega_T\,|\, s_1^*<u_x(x,t)<s_1\}\big|>0,\\
\big|\{(x,t)\in \Omega_T\,|\, s_2<u_x(x,t)<s_2^*\}\big|>0.\end{cases}
\end{equation}
Therefore, these solutions must exhibit instantaneous phase transitions between the two forward  phases, and we will call  such solutions {\em two-phase forward solutions.} In fact, these two-phase forward  solutions will also satisfy a much stronger condition in terms of  the {\em  transition gauge} defined later -- they are {\em almost  transition gauge   invariant.}

We remark that property (\ref{2-phase-condition}) is actually strengthened in Theorem \ref{thm:main-detailed}  so that   the constructed solutions $u$ satisfy that $(u_x,\sigma(u_x))$ stays in both parts of the set $K'$ defined in (\ref{set-K'U'}) below on sets of positive measure in $\Omega_T$ and thus internally give rise to a {\em hysteresis loop}  \cite{V} that is the boundary of the set $U'$ in (\ref{set-K'U'}); this is  typical in the first-order phase transitions similar to  the Ericksen bar problem \cite{Er}.

The paper is organized as follows. Our main results are stated in  Section \ref{sec:main-results} as  Theorem \ref{thm:main-simple} and in a more detailed fashion as Theorem \ref{thm:main-detailed} after the discussion of  the main idea of  a {\em nonlocal} differential inclusion and  the introduction of  {\em subsolutions} of the inclusion and  the notion of  {\em transition gauge} for these subsolutions and for weak solutions of  (\ref{p-main}).  In Section \ref{sec:prelim}, a special subsolution $w^*$ is constructed by solving a modified parabolic problem. Section \ref{admissible-set} defines   a suitable admissible set of subsolutions  so that a pivotal density lemma, Lemma \ref{lem:density-fact}, can be proved later.  In Section \ref{sec:proof-main},  the proof of Theorem \ref{thm:main-detailed} is completed in the framework of Baire's category method with the help of Lemma \ref{lem:density-fact}. As the last part of the paper, the long proof of Lemma \ref{lem:density-fact} is given in  Section \ref{sec:density-proof} following  a technical lemma  in Section \ref{lem:tech}.

In closing this section, we introduce some notations. For a measurable set $E\subset\R^n$ $(n=1,2)$, we denote by $|E|$ its Lebesgue measure. For a matrix $A=(a_{ij})$ in the space $\M^{2\times 2}$ of $2\times 2$ real matrices, we use $|A|$ to denote its Hilbert-Schmidt norm, that is, $|A|=(\sum_{i,j=1}^2 a_{ij}^2)^{1/2}$. For a vector $a=(a_1,a_2)\in\R^2$, $|a|=(\sum_{i=1}^2 a_{i}^2)^{1/2}$ is its Euclidean norm. We mainly follow the notations in the book \cite{LSU} for function spaces, with one exception  that the letter $C$ is used instead of $H$ regarding suitable (parabolic) H\"older spaces.
For integers $k,l\ge 0$ with $2l\le k$, we denote by $C^{k,l}(\bar\Omega_T)$   the space of functions $u\in C^0(\bar\Omega_T)$   such that $\partial^\beta_x \partial^j_t u\in C^0(\bar\Omega_T)$   for all multiindices $|\beta|\le k$ and integers $0\le j\le l$ with $|\beta|+2j\le k$.

\section{Main results}\label{sec:main-results}

To  fix some notation, let $s_1^*\in[0,s_1)$ and $s_2^*\in(s_2,\infty)$ denote the unique numbers with $\sigma(s_1^*)=\sigma(s_2)$ and $\sigma(s_2^*)=\sigma(s_1)$, respectively. For each $r\in[\sigma(s_2),\sigma(s_1)],$ let $s^+_r\in[s_2,s_2^*]$ and $s^-_r\in[s_1^*,s_1]$ be the unique numbers such that $\sigma(s_r^\pm)=r$. (See Figure \ref{fig1}.)

\begin{defn} A function $u\in W^{1,\infty}(\Omega_T)$ is a \emph{Lipschitz solution} to problem (\ref{p-main})-(\ref{p-main-ib-condition}) provided that equality
\begin{equation}\label{def:solution}
\begin{split}
\int_0^s\int_0^1 \big(u\zeta_t-\sigma(u_x) & \zeta_x +(bu_x+cu+f)\zeta\big)\,dxdt \\
& = \int_0^1 \big(u(x,s)\zeta(x,s)-u_0(x)\zeta(x,0)\big)\,dx
\end{split}
\end{equation}
holds for each $\zeta\in C^\infty(\bar\Omega_T)$ and each $s\in[0,T]$. A Lipschitz solution $u$ to  (\ref{p-main})-(\ref{p-main-ib-condition}) satisfying $u_x\in (-\infty,s_1)\cup (s_2,\infty)$ a.e.\,in $\Omega_T$ will be called a {\em forward  solution}.
\end{defn}

Since all solutions $u$ to problem (\ref{p-main})-(\ref{p-main-ib-condition}) considered in this paper will satisfy $u_x\not\in(s_2-\delta_u,s_2+\delta_u)$ a.e. in $\Omega_T$ for some $\delta_u>0$, we have $u_x=0$ if and only if $\sigma(u_x)=0$ even in the case that $\sigma(s_2)=0$; thus equality (\ref{def:solution}) indeed reflects the Neumann boundary condition in (\ref{p-main-ib-condition}).

We now state our  main existence results in the following simplified version.

\begin{thm}\label{thm:main-simple}
For any initial datum $u_0\in C^{2+\alpha}(\bar\Omega)$ satisfying the compatibility condition (\ref{u_0-compatibility-condition}), there exists at least one forward   solution to problem (\ref{p-main})-(\ref{p-main-ib-condition}).  Furthermore, if, in addition, $s_1^*<u_0'(x_0)<s_2^*$ for some $x_0\in\Omega$, then there are infinitely many  forward   solutions $u$  satisfying (\ref{2-phase-condition}).
\end{thm}

The first part of this theorem follows easily from its second part.  For example,   if the initial datum $u_0$ satisfies $u_0'(x_1)>s_1^*$ for some $x_1\in \Omega,$ then by (\ref{u_0-compatibility-condition}) there must be a point $x_0\in\Omega$ such that $u_0'(x_0)\in (s_1^*,s_2^*);$ thus, in this case, the second part of the theorem implies the existence of forward  solutions to problem (\ref{p-main})-(\ref{p-main-ib-condition}).   Next, assume that  $u_0'(x)\le s_1^*$ for all $x\in\Omega.$ We then solve problem (\ref{p-main})-(\ref{p-main-ib-condition}),  with $\sigma$ replaced by a function $\bar\sigma$ on $\R$ whose derivative always lies in between two positive numbers such that it agrees with $\sigma$ on $(-\infty,(s_1^*+s_1)/2]$, to get a classical solution $u$ in $\Omega_T.$ If $u_x(x,t)\le s_1^*$ for all $(x,t)\in \Omega_T$, then $u$ itself is a (classical) forward  solution to  problem (\ref{p-main})-(\ref{p-main-ib-condition})  with the original $\sigma.$ Otherwise,  due to the Neumann boundary condition, we can choose a $0<T'<T$ such that $u_x(x,t)\le (s_1^*+s_1)/2$ for all $(x,t)\in\Omega\times(0,T')$ and that $u_x(x_0,t_0)\in (s_1^*,(s_1^*+s_1)/2)$ for some $(x_0,t_0)\in\Omega\times(0,T')$.    Then we use $u(\cdot,t_0)$ as an initial datum and apply the second part of the theorem to get forward  solutions $\tilde u$ to (\ref{p-main})-(\ref{p-main-ib-condition}) in $\Omega\times (t_0,T)$ with the original $\sigma$. Piecing $u$ and $\tilde u$ together at $t=t_0$, we obtain  forward  solutions to  (\ref{p-main})-(\ref{p-main-ib-condition})  in $\Omega_T$ with the original $\sigma$.

The second part of Theorem \ref{thm:main-simple} follows from our more detailed statement, Theorem \ref{thm:main-detailed}, to be given below. For this, we first recast equation (\ref{p-main}) as a nonlocal differential inclusion and introduce the concept of subsolutions and transition gauge.

\subsection{Nonlocal differential inclusion} Our key approach to problem  (\ref{p-main})-(\ref{p-main-ib-condition}) is   to treat  equation (\ref{p-main})   as a new type of partial differential inclusions involving a \emph{nonlocal} operator.

To ascertain this treatment, we first let
\[
F(x,t) :=\int_0^x f(y,t)\,dy
\]
 and rewrite equation (\ref{p-main})  as
\[
u_t=\big( \sigma(u_x)+ bu + \mathcal{P}_u+ F\big)_x,
\]
where the {\em potential operator} ${\mathcal P}$ at $u$ is defined by
\[
\mathcal{P}_u(x,t) :=\int_0^x\big( c(y,t)-b_y(y,t) \big)u(y,t)\,dy \quad \forall (x,t)\in \Omega_T.
\]
Therefore,  $u\in W^{1,\infty}(\Omega_T)$ is a  solution of (\ref{p-main}) in the sense of distributions if and only if there exists a stream function $v\in W^{1,\infty}(\Omega_T)$ satisfying
\begin{equation}\label{stream-function-v}
v_x=u,\;\; v_t=\sigma(u_x)+bu+\mathcal{P}_u+F  \;\; \mbox{a.e.\,in $\Omega_T$;}
\end{equation}
such a function $v$ is necessarily unique up to a constant.

Next, let $\Sigma_0=\{(s,\sigma(s))\in\R^2\,|\,s\in\R \},$ which is simply  the graph of the function $\sigma$.
For each $(x,t)\in\Omega_T$ and each $u\in W^{1,\infty}(\Omega_T)$, define the matrix set
\[
\Sigma((x,t);u)  = \left\{ \begin{pmatrix} s & c \\ u(x,t) & q \end{pmatrix}\in\M^{2\times 2} \,\Big|\, \begin{array}{l}
                                           c\in\R, \big(s,q-b(x,t)u(x,t) \\
                                           -\mathcal{P}_u (x,t)-F(x,t)\big)\in \Sigma_0
                                         \end{array}
               \right\}.
\]
Then for  a vector function $w=(u,v)\in W^{1,\infty}(\Omega_T;\R^{2})$, system (\ref{stream-function-v}) is equivalent to the following differential inclusion:
\begin{equation}\label{initial-form-matrix-set}
\nabla w(x,t)=\begin{pmatrix} u_x(x,t) & u_t(x,t)\\v_x(x,t)&v_t(x,t)\end{pmatrix}\in\Sigma((x,t);u), \;\;\mbox{a.e. $(x,t)\in\Omega_T$}.
\end{equation}
In this manner, (\ref{p-main}) becomes equivalent to differential inclusion (\ref{initial-form-matrix-set}).

\begin{remk} We remark that the matrix sets in differential inclusion (\ref{initial-form-matrix-set}) are depending not only on the space-time variable $(x,t)$ but also on the input function $u$ in a nonlocal fashion.   So we may call such an inclusion  a \emph{nonlocal} partial differential inclusion, which, to our best knowledge, has not been studied in the well-known literatures on partial differential inclusion problems \cite{DM1,MSy}; however, see \cite{Sy} for a treatment of a nonlocal problem using differential inclusion.
\end{remk}

\subsection{Subsolutions and transition gauge}  Let
\[
\begin{split}
K_0 & = \{ (s,\sigma(s))\in\R^2\,|\, s\in [s_1^*,s_1] \cup [s_2,s_2^*] \}, \\
U_0 & = \{ (s,r)\in\R^2\,|\, \sigma(s_2)<r<\sigma(s_1),\;  s^-_{r}<s<s^+_r \}.
\end{split}
\]
A function  $w=(u,v)\in W^{1,\infty}(\Omega_T;\R^2)$ is called a {\em subsolution} of differential inclusion (\ref{initial-form-matrix-set}) if
\[
v_x=u\;\;\mbox{and}\;\; (u_x, v_t-bu-\mathcal{P}_u-F)\in  \Sigma_0 \cup  {U_0}\;\;\mbox{a.e.\,in $\Omega_T.$}
\]
In this case, we define the {\em transition set} of $w$  by
\begin{equation}\label{sub-1}
O^w=\{(x,t)\in \Omega_T\;|\; (u_x, v_t-bu-\mathcal{P}_u-F)\in   K_0\cup {U_0} \}.
\end{equation}
Note that the transition set $O^w$ may be a null set; that is, $|O^w|=0.$

Let $w=(u,v)$  be a subsolution of (\ref{initial-form-matrix-set}). Define a function $Z_w\colon O^w\to [0,1]$ by
\begin{equation}\label{Zw}
Z_w(x,t) =\frac{u_x(x,t)-S^-_w(x,t)}{S^+_w(x,t) -S^-_w(x,t)},
\end{equation}
where
\[
 S^\pm_w(x,t) = s^\pm_{v_t(x,t)-b(x,t)u(x,t)-\mathcal{P}_u(x,t)-F(x,t)};
 \]
then, for each non-null subset $E$ of $O^w$, we  define the  \emph{transition gauge} of $w$  on $E$  by
\begin{equation}\label{fractional-distance}
\Gamma_{w}^E =\frac{1}{|E|}\int_E Z_w(x,t)  \,dxdt.
\end{equation}
If $E$ is any measurable subset of $\Omega_T$ with $|E|=0$ or $E\subset\Omega_T\setminus O^w$, we define $\Gamma_{w}^E=-1$.
In case of a non-null set $E\subset O^w$, note that $\Gamma_w^E \in [0,1]$ and that $\Gamma_w^E=1$ (0, resp.)  if and only if $(u_x, v_t-bu-\mathcal{P}_u-F)$ belongs to the right-branch (left-branch, resp.) of ${K_0}$ a.e.\,in $E.$ Therefore, $\Gamma_w^E$ measures the tendency of phases of the subsolution $w$ towards the right-branch of $K_0$ on $E$.

A function $u\in W^{1,\infty}(\Omega_T)$ is called a {\em subsolution} of equation (\ref{p-main}) if there exists a function $v\in W^{1,\infty}(\Omega_T)$ so that $w=(u,v)$ is a subsolution  of (\ref{initial-form-matrix-set}).

Note that if $u\in W^{1,\infty}(\Omega_T)$ is a   solution  of (\ref{p-main}) in the sense of distributions, then, with the stream function $v\in W^{1,\infty}(\Omega_T)$ defined by (\ref{stream-function-v}) (unique up to a constant), the function $w=(u,v)$ is a  solution and thus a subsolution  of (\ref{initial-form-matrix-set}). In this case, we define the {\em transition gauge} of $u$ on any measurable set $E\subset\Omega_T$ by
\begin{equation}\label{transition-gauge}
\gamma_u^E=\Gamma_w^E.
\end{equation}

\begin{remk} For the purpose of this paper,  the transition gauge $\gamma_u^E$ is defined only for (weak) solutions $u$ of equation (\ref{p-main}), not for general subsolutions of (\ref{p-main}).  In fact, in what follows, we only deal with the transition gauge $\Gamma_w^E$ for subsolutions $w$ of differential inclusion (\ref{initial-form-matrix-set})  with $|O^w|>0$ on non-null sets $E\subset O^w$.
\end{remk}

\subsection{Two-phase forward  solutions with almost transition gauge invariance} Assume that the initial datum $u_0\in C^{2+\alpha}(\bar\Omega)$ satisfies  the compatibility condition (\ref{u_0-compatibility-condition}) and  the transition condition
 \begin{equation}\label{trans-0}
 \mbox{$s_1^*<u_0'(x_0)<s_2^*$ \; for some $x_0\in\Omega$. }
 \end{equation}
Then our detailed version for the main result, Theorem \ref{thm:main-simple}, can be formulated as follows.

\begin{thm}\label{thm:main-detailed} Let $r_1,r_2$ be any two numbers with $\sigma(s_2)<r_1<r_2<\sigma(s_1)$ and  $s^-_{r_1}<u_0'(x_0)<s^+_{r_2}$.
Then there exists a subsolution $w^*=(u^*,v^*)\in C^{2,1}(\bar\Omega_T)\times C^{3,1}(\bar\Omega_T)$ of differential inclusion (\ref{initial-form-matrix-set}) such that the sets
%\begin{equation}\label{sub-domains}
\[
\begin{split}
\Omega^1_T   = \{(x,t)\in\Omega_T \,|\, u^*_x(x,t)<s^-_{r_1} \}, \; \; &\Omega^3_T   = \{(x,t)\in\Omega_T \,|\, u^*_x(x,t)>s^+_{r_2} \}, \\
\Omega^2_T   = \{(x,t)\in\Omega_T \,|\, s^-_{r_1}&<u^*_x(x,t)<s^+_{r_2} \},\\
\Omega^{0,-}_T  = \{(x,t)\in\Omega_T \,|\, u^*_x(x,t) = s^-_{r_1} \},\;\;
& \Omega^{0,+}_T  = \{(x,t)\in\Omega_T \,|\, u^*_x(x,t)= s^+_{r_2}\}
\end{split}
\]
%\end{equation}
satisfy  $E:=\Omega^2_T\ne\emptyset$, $0<\Gamma_{w*}^E<1$, and
\begin{equation}\label{u-star-near-bdry}
\left\{
\begin{array}{l}
  ((0,\delta^*)\cup(1-\delta^*,1))  \times(0,T)\subset \Omega^1_T\;\;\mbox{for some $\delta^*>0$}, \\
  \partial\Omega^1_T\cap\{(x,0)\,|\,x\in\Omega\}  = \{(x,0)\,|\,x\in\Omega, u_0'(x)<s^-_{r_1}\}, \\
  \partial\Omega^2_T\cap\{(x,0)\,|\,x\in\Omega\}  = \{(x,0)\,|\,x\in\Omega, s^-_{r_1}<u_0'(x)<s^+_{r_2}\}, \\
  \partial\Omega^3_T\cap\{(x,0)\,|\,x\in\Omega\}  = \{(x,0)\,|\,x\in\Omega, u_0'(x)>s^+_{r_2}\}.
\end{array}\right.
\end{equation}
 Furthermore, for each $\epsilon>0$, there exist infinitely many Lipschitz solutions $u$ to problem (\ref{p-main})-(\ref{p-main-ib-condition}) satisfying the following properties:
\begin{itemize}
\item [(a)] $u=u^*$ in $\Omega^1_T\cup\Omega^3_T$, $\|u-u^*\|_{L^\infty(\Omega_T)}<\epsilon$, $\|u_t-u^*_t\|_{L^\infty(\Omega_T)}<\epsilon,$
\item [(b)] $u_x<s^-_{r_1}$ in $\Omega^1_T$, $u_x\in [s^-_{r_1},s^-_{r_2}]\cup [s^+_{r_1},s^+_{r_2}]$ a.e.\,in $\Omega^2_T$, $u_x>s^+_{r_2}$ in $\Omega^3_T$,
\item [(c)] $u_x=s^-_{r_1}$ a.e.\,in $\Omega_T^{0,-}$,   $u_x= s^+_{r_2}$ a.e.\,in $\Omega_T^{0,+},$
\item [(d)] $|F^+_T|=\gamma_u^E|\Omega^2_T|$, $|F^-_T|=(1-\gamma_u^E)|\Omega^2_T|$, and $|\gamma_u^E-\Gamma_{w^*}^{E}|<\epsilon$,
where  $F^\pm_T:=\{(x,t)\in\Omega^2_T\,|\,u_x(x,t)\in[s^\pm_{r_1},s^\pm_{r_2}]\}.$
\end{itemize}
\end{thm}

Clearly, this theorem implies the second part of Theorem \ref{thm:main-simple}  if the number $\epsilon$ is chosen so that
\[
 0<\epsilon<\min\{\Gamma^E_{w^*},1-\Gamma^E_{w^*}\};
\]
thus such functions $u$ are two-phase forward solutions to problem (\ref{p-main})-(\ref{p-main-ib-condition}).

The proof of  the first part of Theorem \ref{thm:main-detailed} is given in the next section; that of the second part will be given in Section \ref{sec:proof-main}.

\section{Construction of subsolution by a modified problem}\label{sec:prelim}
In this section, we assume that $u_0\in C^{2+\alpha}(\bar\Omega)$ satisfies  conditions (\ref{u_0-compatibility-condition}) and  (\ref{trans-0}).  We also assume that $r_1,r_2$ are any two numbers such that
\[
\sigma(s_2)<r_1<r_2<\sigma(s_1), \quad s^-_{r_1}<u_0'(x_0)<s^+_{r_2}.
\]

\subsection{Modified problem} By elementary calculus with (\ref{flux-condition}), we can construct  a function $\tilde\sigma\in C^{1+\alpha}(\R)$ satisfying
\begin{equation}\label{modified-sigma}
\left\{
\begin{array}{l}
  \tilde\sigma=\sigma\;\;\mbox{on $(-\infty,s^-_{r_1}]\cup[s^+_{r_2},\infty)$,} \\
  \mbox{$\tilde\sigma<\sigma$ on $(s^-_{r_1},s^-_{r_2}]$, $\tilde\sigma>\sigma$ on $[s^+_{r_1},s^+_{r_2})$, and} \\
  \tilde\lambda\le \tilde\sigma' \le\tilde\Lambda\;\;\mbox{in $\R$},
\end{array}\right.
\end{equation}
where $\tilde\Lambda\ge\tilde\lambda>0$ are some constants. (See Figure \ref{fig2}.)

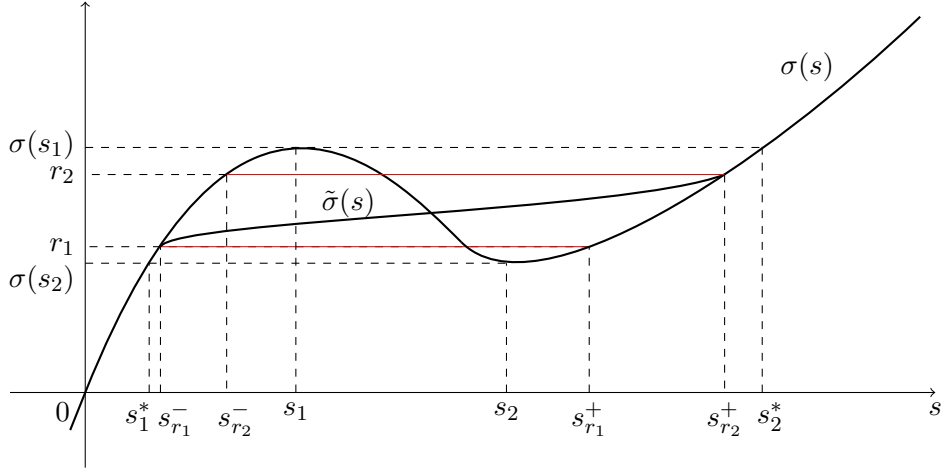
\begin{figure}[ht]
\begin{center}
\begin{tikzpicture}[scale = 1]
    \draw[->] (-1,0) -- (11.3,0);
	\draw[->] (0,-1) -- (0,5.2);
 \draw[dashed] (0,1.72)--(5.6,1.72);
 \draw[dashed] (0,3.26)--(9.05,3.26);
 \draw[dashed] (9,0)--(9,3.3);
    \draw[dashed] (5.6, 0)  --  (5.6, 1.72) ;
   \draw[dashed] (0.85, 0)  --  (0.85, 1.72) ;
\draw[dashed] (1, 0)  --  (1, 1.94) ;
\draw[dashed](6.7, 1.94)-- (0,1.94);
\draw[red] (6.7, 1.94)--(1,1.94) ;
\draw[dashed] (6.7,0)  --  (6.7, 1.94) ;
   \draw[dashed] (2.8, 0)  --  (2.8, 3.3) ;
   \draw[thick] (-0.2,-0.5)  --  (0,0) ;
	\draw[thick]   (0, 0) .. controls (2,5) and  (4, 3)   ..(5,2);
          \draw[thick]   (1, 1.94) .. controls (1.2,2.3) and  (7.2, 2.4)   ..(8.5,2.9);
          %\draw[thick]   (10.0, 4.03) .. controls (10.3,4.3) and  (10.8, 4.4)   ..(11.1,4.5);
\draw[thick]   (5, 2) .. controls  (6, 1) and (9,3) ..(11.1, 5 );
%\draw[thick]   (0.9,1.5) .. controls  (1,1.6) and (2,2.6) ..(11,2.8);
	\draw (11.3,0) node[below] {$s$};
    \draw (-0.3,0) node[below] {{$0$}};
    \draw (9.6,4) node[above] {$\sigma(s)$};
 \draw (0, 3.3) node[left] {$\sigma(s_1)$};
 \draw (0, 1.5) node[left] {$\sigma(s_2)$};
 \draw (0, 1.94) node[left] {$r_1$};
 \draw (0, 2.9) node[left] {$r_2$};
% \draw[dashed] (0,1.5)--(9.2, 1.5);
\draw[dashed] (1.88,2.9)--(1.88,0);
\draw[dashed] (1.88,2.9)--(0,2.9);
\draw[red] (8.5,2.9)--(1.88,2.9);
\draw[dashed] (8.5,2.9)--(8.5, 0);
   \draw (5.6, 0) node[below] {$s_2$};
   \draw (3.5, 2.2) node[above] {$\tilde\sigma(s)$};
   \draw (2.8, 0) node[below] {$s_1$};
   \draw (0.7, 0) node[below] {$s_1^*$};
   \draw (1.2, 0) node[below] {$s_{r_1}^-$};
     \draw (2, 0) node[below] {$s_{r_2}^-$};
   \draw (9.1, 0) node[below] {$s_2^*$};
   %\draw (10.0, 0) node[below] {$s_3$};
   \draw (8.5, 0) node[below] {$s_{r_2}^+$};
   \draw (6.7, 0) node[below] {$s_{r_1}^+$};
% \draw[dashed] (0.9,0)--(0.9, 1.5);
    \end{tikzpicture}
%\end{center}
\caption{\emph{The modified function  $\tilde\sigma$,   constructed as in (\ref{modified-sigma}),   is strictly increasing on $\R$ and agrees with $\sigma$ on $(-\infty,s^-_{r_1}]\cup[s^+_{r_2},\infty)$. The set $K'$ defined in (\ref{set-K'U'}) below consists of the two pieces of the graph of $\sigma$ over $[s_{r_1}^-,s_{r_2}^-]$ and $[s_{r_1}^+,s_{r_2}^+],$ while the set $U'$ in (\ref{set-K'U'})  is the open set bounded by $K'$ and the two parallel  line segments (in red)  joining  the endpoints of $K'.$ The boundary $\partial U'$ of $U'$ represents a hysteresis loop.}
}
\label{fig2}
\end{center}
\end{figure}

From the standard parabolic theory \cite[Theorem 13.24]{Ln}, we obtain a unique classical solution $u^*\in C^{2+\alpha,1+\frac{\alpha}{2}}(\bar\Omega_T)$ to the initial-boundary value problem
\begin{equation}\label{modified-problem}
\left\{
\begin{array}{ll}
  u_t^*=(\tilde\sigma(u^*_x))_x+b(x,t)u^*_x+c(x,t)u^*+f(x,t) & \mbox{in $\Omega_T$}, \\
  u^*=u_0 & \mbox{on $\Omega\times\{t=0\}$}, \\
  u^*_x=0 & \mbox{on $\partial\Omega\times(0,T)$.}
\end{array}\right.
\end{equation}
We then have from (\ref{coefficient-condition}) that
\begin{equation}\label{equation-u-star}
\begin{split}
(\tilde\sigma(u^*_x)+b u^*+\mathcal{P}_{u^*}+F)_x & = \tilde\sigma(u^*_x)_x + b_x u^* + b u^*_x + d(t)u^*+f \\
& = \tilde\sigma(u^*_x)_x  + b u^*_x + c u^* + f = u^*_t\;\;\;\mbox{in $\Omega_T.$}
\end{split}
\end{equation}

\subsection{Separation of domain} Define the  subsets of $\Omega_T$ by
\begin{equation}\label{sub-domains}
\begin{split}
\Omega^1_T & = \{(x,t)\in\Omega_T \,|\, u^*_x(x,t)<s^-_{r_1} \}, \\
\Omega^2_T & = \{(x,t)\in\Omega_T \,|\, s^-_{r_1}<u^*_x(x,t)<s^+_{r_2} \}, \\
\Omega^3_T & = \{(x,t)\in\Omega_T \,|\, u^*_x(x,t)>s^+_{r_2} \}, \\
\Omega^{0,-}_T & = \{(x,t)\in\Omega_T \,|\, u^*_x(x,t) = s^-_{r_1} \}, \\
\Omega^{0,+}_T & = \{(x,t)\in\Omega_T \,|\, u^*_x(x,t)= s^+_{r_2}\}.
\end{split}
\end{equation}
Note that  $\Omega^2_T\ne\emptyset$ as $s^-_{r_1}<u_0'(x_0)<s^+_{r_2}$. From (\ref{modified-problem}), we easily have
\begin{equation}\label{u-star-near-bdry}
\left\{
\begin{array}{l}
  ((0,\delta^*)\cup(1-\delta^*,1))  \times(0,T)\subset \Omega^1_T\;\;\mbox{for some $\delta^*>0$}, \\
  \partial\Omega^1_T\cap\{(x,0)\,|\,x\in\Omega\}  = \{(x,0)\,|\,x\in\Omega, u_0'(x)<s^-_{r_1}\}, \\
  \partial\Omega^2_T\cap\{(x,0)\,|\,x\in\Omega\}  = \{(x,0)\,|\,x\in\Omega, s^-_{r_1}<u_0'(x)<s^+_{r_2}\}, \\
  \partial\Omega^3_T\cap\{(x,0)\,|\,x\in\Omega\}  = \{(x,0)\,|\,x\in\Omega, u_0'(x)>s^+_{r_2}\}.
\end{array}\right.
\end{equation}

\subsection{Construction of subsolution} Define an  \emph{auxiliary function} $v^*$  by
\[
\begin{split}
v^*(x,t)= & \int_0^t \big( \tilde\sigma(u^*_x(x,\tau)) +b(x,\tau)u^*(x,\tau)+\mathcal{P}_{u^*}(x,\tau)+F(x,\tau) \big)\,d\tau \\
& + \int_0^x u^*(y,0)\,dy \quad \forall (x,t)\in\Omega_T.
\end{split}
\]
From (\ref{equation-u-star}), the function $w^*:=(u^*,v^*)$ satisfies that in $\Omega_T$,
\begin{equation}\label{equation-w-star}
\left\{
\begin{array}{l}
  v^*_t-b u^*-\mathcal{P}_{u^*}-F=\tilde\sigma(u_x^*), \\
  v^*_x=u^*,
\end{array}\right.
\end{equation}
and hence from (\ref{modified-sigma}), $w^*\in C^{2,1}(\bar\Omega_T)\times C^{3,1}(\bar\Omega_T)$ is a subsolution of (\ref{initial-form-matrix-set}). Clearly,  $E=\Omega_T^2\subset O^{w^*}$, and it follows
from (\ref{equation-w-star}) that  the number $\Gamma_{w^*}^E$   is given by
\[
\Gamma_{w^*}^E=\frac{1}{|\Omega^2_T|}\int_{\Omega^2_T} \frac{u_x^*-s^-_{\tilde\sigma(u^*_x)}}{s^+_{\tilde\sigma(u^*_x)}-s^-_{\tilde\sigma(u^*_x)}}\,dxdt,
\]
which lies in $(0,1)$.

\section{Admissible set and density lemma}\label{admissible-set}
We use the same  notation as in the previous section and, for simplicity, in what follows, we say that an open set $G\subset \R^2$  is \emph{regular} if $|\partial G|=0$.

\subsection{Related matrix sets} Define the sets (see Figure \ref{fig2})
\begin{equation}\label{set-K'U'}
\begin{split}
K' & = \{ (s,\sigma(s))\in\R^2\,|\, s\in[s^-_{r_1},s^-_{r_2}]\cup[s^+_{r_1},s^+_{r_2}] \}, \\
U' & = \{ (s,r)\in\R^2\,|\, r_1<r<r_2, s^-_{r}<s<s^+_r \}.
\end{split}
\end{equation}
For each $(x,t)\in\Omega_T$ and each $u\in W^{1,\infty}(\Omega_T)$, define the matrix sets
\[
\begin{split}
K((x,t);u) & = \left\{ \begin{pmatrix} s & c \\ u(x,t) & q \end{pmatrix}\in\M^{2\times 2} \,\Big|\, \begin{array}{l}
                                           |c|\le m^*, \big(s,q-b(x,t)u(x,t) \\
                                           -\mathcal{P}_u (x,t)-F(x,t)\big)\in K'
                                         \end{array}
               \right\}, \\
U((x,t);u) & = \left\{ \begin{pmatrix} s & c \\ u(x,t) & q \end{pmatrix}\in\M^{2\times 2} \,\Big|\, \begin{array}{l}
                                           |c|< m^*, \big(s,q-b(x,t)u(x,t) \\
                                           -\mathcal{P}_u (x,t) -F(x,t)\big)\in U'
                                         \end{array}
               \right\},
\end{split}
\]
where $m^*:=\|u^*_t\|_{L^\infty(\Omega_T)}+1>0.$

Note that $K((x,t);u)$ is a compact subset of $\M^{2\times2}$ for each $(x,t)\in\Omega_T$ and each $u\in W^{1,\infty}(\Omega_T)$.

\subsection{Admissible  set of subsolutions}  For any fixed $\epsilon >0$,
 we define a set $\mathcal A$ of  \emph{admissible functions}    by
\[
\mathcal{A}=\left\{ \begin{array}{l}
                      w=(u,v)\;\mbox{belongs to} \\
                      C^{2,1}(\bar\Omega_T)\times C^{3,1}(\bar\Omega_T)
                    \end{array}
  \,\Bigg|\, \begin{array}{l}
                                                               \mbox{$w= w^*$ in $\Omega_T\setminus \bar\Omega^w_T$ for some regular }   \\
                                                               \mbox{open set $\Omega^w_T\subset\subset \Omega^2_T=E$,} \\
                                                               \|u-u^*\|_\infty <\epsilon/2,\;\; \|u_t-u^*_t\|_\infty <\epsilon/2,\\
                                                               \nabla w(x,t)\in U((x,t);u)\;\forall(x,t) \in \Omega^2_T, \\
                                                               \mbox{$\mathcal{P}_u= \mathcal{P}_{u^*}$ in $\Omega_T\setminus\bar \Omega^w_T$, $|\Gamma_w^E- \Gamma_{w^*}^E|< {\epsilon}/{2}$}
                                                             \end{array}
 \right\},
\]
where we use $\|\cdot\|_\infty=\|\cdot\|_{L^\infty(\Omega_T)}$ here and below.
We easily see that $\mathcal{A}\ne\emptyset$ as  $w^*\in\mathcal{A}$ and that each function $w$ in $\mathcal A$ is a subsolution of (\ref{initial-form-matrix-set}) with its transition set $O^w$ containing $E=\Omega_T^2.$

%  Let
%\begin{equation}\label{d*}
%0<\epsilon<\min\{\Gamma_{w^*},1-\Gamma_{w^*}\}.
%\end{equation}

%For a later use, we record here some $L^\infty$-bounds for functions in $\mathcal{A}$ and their partial derivatives. Choose any $w=(u,v)\in\mathcal{A}$. The most obvious one is
%\[
%\|u_t\|_{L^\infty(\Omega_T)}\le m^*.
%\]
%Since
%\[
%u(x,t)=\int_0^t u_\tau(x,\tau)\,d\tau + u_0(x)\;\;\forall(x,t)\in\Omega_T,
%\]
%we get
%\[
%\|v_x\|_{L^\infty(\Omega_T)}=\|u\|_{L^\infty(\Omega_T)}\le Tm^* +\|u_0\|_{L^\infty(\Omega)}.
%\]
%It is also easy to see that
%\[
%\|u_x\|_{L^\infty(\Omega_T)} \le \|u^*_x\|_{L^\infty(\Omega_T)} + s^+_{r_2},
%\]
%\[
%\|v_t\|_{L^\infty(\Omega_T)} \le \|v^*_t\|_{L^\infty(\Omega_T)} + r_2+ \|u\|_{L^\infty(\Omega_T)}(\|b\|_{L^\infty(\Omega_T)}+\|d\|_{L^\infty(\Omega)})
%\]
%\[
%\le \|v^*_t\|_{L^\infty(\Omega_T)} + r_2+ (Tm^* +\|u_0\|_{L^\infty(\Omega)})(\|b\|_{L^\infty(\Omega_T)}+\|d\|_{L^\infty(\Omega)})
%\]

For each $ \delta>0$, we define the set $\mathcal{A}_\delta$ of \emph{$\delta$-approximating functions}  by
\[
\mathcal{A}_\delta=\left\{ w\in\mathcal{A} \,\Big|\,  \int_{\Omega^2_T} \mathrm{dist} \big(\nabla w(x,t),K((x,t);u)\big)\,dxdt\le\delta|\Omega^2_T| \right\}.
\]

\subsection{Density lemma} The completion of the proof of Theorem \ref{thm:main-detailed} relies on  the following  pivotal  density lemma.
\begin{lem}[\bf Density Lemma] \label{lem:density-fact}
For each $\delta>0$, $\mathcal A_\delta$ is dense in $\mathcal A$ under the $L^\infty$-norm.
\end{lem}
The proof of this lemma is long and will be given in the last part of the paper, Section \ref{sec:density-proof}.

\section{Proof of Theorem \ref{thm:main-detailed}}\label{sec:proof-main}

We continue to use the same notation as above. The first part of Theorem \ref{thm:main-detailed} has already been proved in Section \ref{sec:prelim}. In this section we carry out the remaining part of the proof in the functional framework of  Baire's category method based on Lemma  \ref{lem:density-fact}.

\subsection{Baire's category method}
Let $\mathcal{X}$ denote the closure of $\mathcal A$ in the space $L^\infty(\Omega_T;\R^2)$; then $(\mathcal{X},L^\infty)$ is a nonempty complete metric space. It is clear that $\mathcal A$ is a bounded subset of $w^*+W^{1,\infty}_0(\Omega_T;\R^2)$, and so $\mathcal{X}\subset  w^*+W^{1,\infty}_0(\Omega_T;\R^2)$.

Since the space-time gradient operator $\nabla:\mathcal{X}\to L^1(\Omega_T;\M^{2\times 2})$ is a Baire-one map \cite[Proposition 10.17]{Da}, it follows from the Baire Category Theorem \cite[Theorem 10.15]{Da} that the set of points of discontinuity of the operator $\nabla$, say $\mathcal{D}_\nabla,$ is a set of the first category; thus the set of points at which $\nabla$ is continuous, that is, $\mathcal{C}_\nabla:=\mathcal{X}\setminus\mathcal{D}_\nabla$, is dense in $\mathcal{X}$.

We now proceed to complete the proof of Theorem \ref{thm:main-detailed} assuming the validity of  Lemma \ref{lem:density-fact}.

\begin{pro}  Let $w=(u,v)\in\mathcal{C}_\nabla$. Then $u$ is a Lipschitz solution to problem (\ref{p-main})-(\ref{p-main-ib-condition}) satisfying (a)--(d) of Theorem \ref{thm:main-detailed}.
\end{pro}

\begin{proof} By the density of $\mathcal{A}$ in $\mathcal{X}$, we can choose a sequence $\{\tilde w_j\}_{j\in\N}$ in $\mathcal{A}$ so that $\|\tilde{w}_j-w\|_\infty\to 0$ as $j\to\infty$. By Lemma \ref{lem:density-fact}, for each $j\in\N$, we can choose a function $w_j=(u_j,v_j)\in\mathcal{A}_{1/j}$ with $\|w_j-\tilde w_j\|_\infty<1/j$. As $w_j\to w$ in $\mathcal{X}$ and $w\in\mathcal{C}_\nabla$, we  have that $\nabla w_j\to\nabla w$ in $L^1(\Omega_T;\M^{2\times 2})$ and thus a.e.\,in $\Omega_T$ after passing to a subsequence if necessary.

Let $j\in\N$. Since $w_j\in\mathcal{A}_{1/j}\subset\mathcal{A}$, we have the following:
\begin{equation}\label{property-w_j}
\left\{
\begin{array}{l}
  \mbox{$w_j= w^*$ in $\Omega_T\setminus\bar\Omega^{w_j}_T$ for some regular open set $\Omega^{w_j}_T\subset\subset\Omega^2_T$,} \\
  \nabla w_j(x,t)\in U((x,t);u_j)\;\;\forall(x,t)\in E= \Omega^2_T, \\
  \mbox{$\mathcal{P}_{u_j}=\mathcal{P}_{u^*}$ in $\Omega_T\setminus\bar\Omega^{w_j}_T$,}\\
    \|u_j-u^*\|_\infty <\frac{\epsilon}{2},\; \|(u_j)_t-u^*_t\|_\infty <\frac{\epsilon}{2}, \; |\Gamma^E_{w_j}-\Gamma^E_{w^*}|<\frac{\epsilon}{2}, \;\mbox{and} \\
  \int_{\Omega^2_T}\mathrm{dist} \big( \nabla w_j(x,t), K((x,t);u_j) \big)\,dxdt\le\frac{1}{j}|\Omega^2_T|.
\end{array}\right.
\end{equation}
From this and (\ref{equation-w-star}), we easily see that
\begin{equation}\label{property-w_j-sequel}
\left\{
\begin{array}{l}
\mbox{$(u_j)_x=s^-_{r_1}$   in $\Omega_T^{0,-}$,   $(u_j)_x= s^+_{r_2}$   in $\Omega_T^{0,+},$} \\
  u_j= u^*\;\;\mbox{in $\Omega^1_T\cup\Omega^3_T$,   and} \\
  (v_j)_x=u_j\;\;\mbox{in $\Omega_T$}.
\end{array}\right.
\end{equation}
Also, from (\ref{modified-sigma}), (\ref{equation-w-star}), (\ref{property-w_j}) and the definition of $\Omega^2_T$, it follows that in $\Omega_T\setminus\Omega^2_T$,
\[
(v_j)_t-b u_j-\mathcal{P}_{u_j}-F = v^*_t-b u^*-\mathcal{P}_{u^*} -F = \tilde\sigma(u^*_x) = \sigma((u_j)_x).
\]
Now, let $j\to\infty$; since $\nabla w_j\to\nabla w$  a.e.\,in $\Omega_T$, we have
\begin{equation}\label{w-fn-classical-part}
\left\{
\begin{array}{l}
  \mbox{$u_x=s^-_{r_1}$ a.e. in $\Omega_T^{0,-}$,   $u_x= s^+_{r_2}$ a.e. in $\Omega_T^{0,+},$}  \\
   u= u^*\;\;\mbox{in $\Omega^1\cup\Omega^3_T$},\;\; v_x=u\;\;\mbox{a.e. in $\Omega_T,$}\\
  \|u-u^*\|_\infty \le \frac{\epsilon}{2},\; \|u_t-u^*_t\|_\infty \le \frac{\epsilon}{2}, \; |\Gamma^E_{w}-\Gamma^E_{w^*}|\le \frac{\epsilon}{2}, \;\mbox{and} \\
   v_t-bu-\mathcal{P}_{u}-F = \sigma(u_x)\;\;\mbox{a.e. in $\Omega_T\setminus\Omega^2_T$}.
\end{array}\right.
\end{equation}
Here, we see that (a) and (c) of Theorem \ref{thm:main-detailed} are satisfied.
Also, it follows from (\ref{property-w_j}), the continuity of the distance function and the Dominated Convergence Theorem that
\[
\nabla w(x,t)\in K((x,t);u)\;\;\mbox{a.e. $(x,t)\in\Omega^2_T$}.
\]
This inclusion implies that a.e. in $\Omega^2_T$,
\begin{equation}\label{w-fn-Lipschitz-part}
\left\{
\begin{array}{l}
  u_x\in[s^-_{r_1},s^-_{r_2}]\cup[s^+_{r_1},s^+_{r_2}],  \\
  v_t-bu-\mathcal{P}_{u}-F=\sigma(u_x).
  %v_x=u.
\end{array}\right.
\end{equation}
Thus (b) of Theorem \ref{thm:main-detailed} follows from (\ref{w-fn-classical-part}), (\ref{w-fn-Lipschitz-part}) and the definition of $\Omega^1_T$ and $\Omega^3_T$. From (\ref{property-w_j}), (\ref{w-fn-classical-part}) and (\ref{w-fn-Lipschitz-part}), we have
\begin{equation}\label{w-fn-global}
\mathcal{P}_{u}=\mathcal{P}_{u^*}\;\;\mbox{on $\bar\Omega_T\setminus\Omega^2_T$}, \quad v_t-bu-\mathcal{P}_{u} -F=\sigma(u_x)\;\;\mbox{a.e. in $\Omega_T$}.
\end{equation}
From (\ref{w-fn-Lipschitz-part}) and the validity of (b), we also have
\[
\Gamma^E_w   = \frac{1}{|\Omega^2_T|} \int_{F^+_T\cup F^-_T}\frac{u_x-s^-_{\sigma(u_x)}}{s^+_{\sigma(u_x)}-s^ -_{\sigma(u_x)}}\,dxdt =\frac{|F^+_T|}{|\Omega^2_T|};
\]
that is,
\[
|F^+_T|=\Gamma^E_w|\Omega^2_T|,\quad |F^-_T|=(1-\Gamma^E_w)|\Omega^2_T|,
\]
where $|\Gamma^E_w-\Gamma^E_{w^*}|\le \frac{\epsilon}{2}<\epsilon$. Since $u$ is a Lipschitz solution of equation (\ref{p-main}) as we check below, it follows from (\ref{w-fn-classical-part}), (\ref{w-fn-global}) and the definition of $\gamma^E_u$ in (\ref{transition-gauge}) that $\gamma^E_u=\Gamma^E_w$. Thus (d) holds.

Finally, to verify that $u$ is a Lipschitz solution to (\ref{p-main})-(\ref{p-main-ib-condition}),
let $j\in\N$, $\zeta\in C^\infty(\bar\Omega_T)$ and $s\in[0,T]$.
Note from (\ref{property-w_j-sequel}) and $w_j=w^*$ on $\partial\Omega_T$ that
\[
\begin{split}
\int_0^s & \int_0^1 u_j\zeta_t\,dxdt  =  \int_0^s\int_0^1 (v_j)_x\zeta_t\,dxdt \\
= & -\int_0^s\int_0^1 (v_j)_{xt}\zeta\,dxdt +\int_0^1 \big((v_j)_x(x,s)\zeta(x,s)-(v_j)_x(x,0)\zeta(x,0)\big)\,dx \\
= & \int_0^s\int_0^1 (v_j)_t\zeta_{x}\,dxdt - \int_0^s \big((v_j)_t(1,t)\zeta(1,t)- (v_j)_t(0,t)\zeta(0,t)\big)\,dt \\
& +\int_0^1 \big((v_j)_x(x,s)\zeta(x,s)-(v_j)_x(x,0)\zeta(x,0)\big)\,dx \\
= & \int_0^s\int_0^1 (v_j)_t\zeta_{x}\,dxdt - \int_0^s \big(v^*_t(1,t)\zeta(1,t)- v^*_t(0,t)\zeta(0,t)\big)\,dt \\
& +\int_0^1 \big(u_j(x,s)\zeta(x,s)-u_0(x)\zeta(x,0)\big)\,dx.
\end{split}
\]
Let $j\to\infty$; then by (\ref{coefficient-condition}) and (\ref{w-fn-global}),
\[
\begin{split}
& \int_0^s  \int_0^1  u\zeta_t\,dxdt =  \int_0^s\int_0^1 v_t\zeta_x\,dxdt - \int_0^s \big(v^*_t(1,t)\zeta(1,t)- v^*_t(0,t)\zeta(0,t)\big)\,dt \\
& \quad\quad\quad\quad\quad\quad\quad\;\;\;\, +\int_0^1 \big(u(x,s)\zeta(x,s)-u_0(x)\zeta(x,0)\big)\,dx \\
= & \int_0^s\int_0^1 \big(bu+\mathcal{P}_u+F+\sigma(u_x)\big)\zeta_x\,dxdt \\
& - \int_0^s \big(v^*_t(1,t)\zeta(1,t)- v^*_t(0,t)\zeta(0,t)\big)\,dt   +\int_0^1 \big(u(x,s)\zeta(x,s)-u_0(x)\zeta(x,0)\big)\,dx \\
= & \int_0^s\int_0^1 \big(\sigma(u_x)\zeta_x -(b u_x+c u+f)\zeta \big)\,dxdt + \int_0^1 \big(u(x,s)\zeta(x,s)-u_0(x)\zeta(x,0)\big)\,dx, \\
\end{split}
\]
where the last equality comes from the observation through (\ref{equation-w-star}) that
\[
\begin{split}
\int_0^s \big( & v^*_t(1,t)\zeta(1,t)- v^*_t(0,t)\zeta(0,t)\big)\,dt \\
& = \int_0^s \Big( \big(b(1,t)u^*(1,t)+\mathcal{P}_{u^*}(1,t)+F(1,t) +\tilde\sigma(u^*_x(1,t))\big) \zeta(1,t) \\
&\quad - \big(b(0,t)u^*(0,t)+\mathcal{P}_{u^*}(0,t)+F(0,t) +\tilde\sigma(u^*_x(0,t))\big) \zeta(0,t)\Big)\,dt \\
& = \int_0^s \Big( \big(b(1,t)u(1,t)+\mathcal{P}_{u}(1,t)+F(1,t) \big) \zeta(1,t) \\
&\quad - \big(b(0,t)u(0,t)+\mathcal{P}_{u}(0,t)+F(0,t) \big) \zeta(0,t)\Big)\,dt.
\end{split}
\]
Thus $u$ satisfies (\ref{def:solution}) and hence is a Lipschitz solution to (\ref{p-main})-(\ref{p-main-ib-condition}).
\end{proof}

\subsection{Completion of proof of Theorem \ref{thm:main-detailed}}
Having shown that the first component $u$ of each function $w=(u,v)\in\mathcal{C}_\nabla$ is a Lipschitz solution to problem (\ref{p-main})-(\ref{p-main-ib-condition}) satisfying (a)--(d), it remains to verify that $\mathcal{C}_\nabla$ contains infinitely many functions and that no two different functions in $\mathcal{C}_\nabla$ have the first component that are equal. First, suppose on the contrary that $\mathcal{C}_\nabla$ has only finitely many functions. Then
\[
\mathcal{C}_\nabla=\bar{\mathcal{C}}_\nabla=\mathcal{X}=\bar{\mathcal{A}} =\mathcal{A}\ni w^*=(u^*,v^*).
\]
By the above argument, $u^*$ becomes a Lipschitz solution to  (\ref{p-main})-(\ref{p-main-ib-condition}) satisfying (a)--(d); this is clearly a contradiction. Therefore, $\mathcal{C}_\nabla$ contains infinitely many functions. Next, let $w_1=(u_1,v_1),w_2=(u_2,v_2)\in\mathcal{C}_\nabla.$ It suffices to show that
\[
u_1=u_2\;\;\mbox{in $\Omega_T$}\;\;\Longleftrightarrow\;\; v_1=v_2\;\;\mbox{in $\Omega_T$}.
\]
If $u_1=u_2$ in $\Omega_T$, then by (\ref{w-fn-classical-part}), we have $(v_1)_x=u_1=u_2=(v_2)_x$ a.e. in $\Omega_T$. As $v_1=v_2=v^*$ on $\partial\Omega_T$, we thus have $v_1=v_2$ in $\Omega_T$. The converse is also easy to show; we omit this.

The proof of Theorem \ref{thm:main-detailed} is now complete, of course, upon assuming the validity of Lemma \ref{lem:density-fact}.

\section{A technical lemma}\label{lem:tech} In this section, we equip with an important but technical tool for local patching to be used in the proof of the density lemma, Lemma \ref{lem:density-fact}.  The following result  is a refinement of the $(1+1)$-dimensional version of a combination of \cite[Theorem 2.3 and Lemma 4.5]{KY1}.

\begin{lem}\label{lem:rank-1-smooth-approx}
Let $Q=(x_1,x_2)\times(t_1,t_2)\subset\R^{2}$ be an open rectangle, where $x_2>x_1$ and $t_2>t_1$ are fixed reals.
Given any $\lambda_1>0$, $\lambda_2>0$ and $\epsilon>0$, there exists a function $\omega=(\varphi,\psi)\in C^\infty_c(Q;\R^2)$ such that
\begin{itemize}
\item[(a)] $\|\omega\|_{L^\infty(Q)}<\epsilon$, $\|\varphi_t\|_{L^\infty(Q)}<\epsilon$, $\|\psi_t\|_{L^\infty(Q)}<\epsilon$,
\item[(b)] $-\lambda_1\le\varphi_x\le\lambda_2$ in $Q$,
\item[(c)] $\left\{\begin{array}{l}
              \big| |\{  (x,t)\in Q\, |\, \varphi_x(x,t) =-\lambda_1 \} |-\frac{\lambda_2}{\lambda_1+\lambda_2}|Q|\big| < \epsilon,\\
              \big| |\{  (x,t)\in Q\, |\, \varphi_x(x,t) =\lambda_2 \} |-\frac{\lambda_1}{\lambda_1+\lambda_2}|Q|\big| < \epsilon,
            \end{array}\right.
$
\item[(d)] $\psi_x=\varphi$ in $Q$, and
\item[(e)] $\int_{x_1}^{x_2}\varphi(x,t)\,dx=0$  for all $t_1<t<t_2$.
\end{itemize}
\end{lem}

\begin{proof}
Let $\tau$ and $\nu$ be sufficiently small positive numbers to be specified below. We first choose a function $h_\tau\in C^\infty_c(t_1,t_2)$ so that
\[
\mbox{$0\le h_\tau\le 1$\;\;in $(t_1,t_2)$\;\;and\;\;$h_\tau= 1$\;\;on $[t_1+\tau,t_2-\tau]$}.
\]
We also choose a (highly oscillating) function $f_\nu\in C^\infty_c(x_1,x_2)$ such that
\[
\left\{
\begin{array}{l}
  -\lambda_1\le f'_\nu \le \lambda_2\;\;\mbox{in $(x_1,x_2)$}, \\
  \big||\{x\in(x_1,x_2)\,|\, f_\nu'(x)=-\lambda_1 \} |-\frac{\lambda_2}{\lambda_1+\lambda_2}(x_2-x_1)\big|<\nu, \\
  \big||\{x\in(x_1,x_2)\,|\, f_\nu'(x)=\lambda_2 \} |-\frac{\lambda_1}{\lambda_1+\lambda_2}(x_2-x_1)\big|<\nu, \\
  \int_{x_1}^{x_2} f_\nu(x)\,dx=0, \\
  \|f_{\nu}\|_{L^\infty(x_1,x_2)}<\nu.
\end{array}\right.
\]

Define $\varphi(x,t)=\varphi_{\tau,\nu}(x,t)=h_\tau(t)f_\nu(x)$ for $(x,t)\in Q$; then $\varphi\in C^\infty_c(Q)$ and $\int_{x_1}^{x_2}\varphi(x,t)\,dx=0$ for all $t_1<t<t_2$; so (e) holds.
Note also that $\|\varphi\|_{L^\infty(Q)}<\nu$, $\|\varphi_t\|_{L^\infty(Q)}<\nu\|h_\tau'\|_{L^\infty(t_1,t_2)}$, $-\lambda_1\le\varphi_x\le\lambda_2$ in $Q$,
\[
\begin{split}
\bigg| |\{ & (x,t)\in Q\, |\, \varphi_x(x,t) =-\lambda_1 \} |-\frac{\lambda_2}{\lambda_1+\lambda_2}|Q|\bigg| \\
\le & \big| |\{(x,t)\in Q\,|\, \varphi_x(x,t)  =-\lambda_1 \}| - |\{x\in(x_1,x_2)\,|\, f_\nu'(x)=-\lambda_1 \}\times(t_1,t_2)|\big| \\
& + \bigg||\{x\in(x_1,x_2)\,|\, f_\nu'(x)=-\lambda_1 \}\times(t_1,t_2) -\frac{\lambda_2}{\lambda_1+\lambda_2}|Q|\bigg| \\
< & 2\tau(x_2-x_1)+\nu(t_2-t_1),\;\;\mbox{and}
\end{split}
\]
\[
\bigg| |\{  (x,t)\in Q\, |\, \varphi_x(x,t) =\lambda_2 \} |-\frac{\lambda_1}{\lambda_1+\lambda_2}|Q|\bigg| < 2\tau(x_2-x_1)+\nu(t_2-t_1);
\]
here, we have (b).

Next, define $\psi(x,t)=\psi_{\tau,\nu}(x,t)=\int_{x_1}^x \varphi(y,t)\,dy$ for $(x,t)\in Q$. Then it is easily checked that $\psi\in C^\infty_c(Q)$ (by (e)), $\|\psi\|_{L^\infty(Q)}<\nu(x_2-x_1)$, $\psi_x=\varphi$ in $Q$, and $\|\psi_t\|_{L^\infty(Q)}<\nu(x_2-x_1)\|h_\tau'\|_{L^\infty(t_1,t_2)}$; thus $\omega:=(\varphi,\psi)\in C^\infty_c(Q;\R^2)$, and (d) is satisfied.

In the above context, for any given $\epsilon>0$, we first choose a $\tau>0$ so small that $2\tau(x_2-x_1)<\epsilon/2.$ We then choose a $\nu>0$ in such a way that
\[
\nu(x_2-x_1+1)<\epsilon,\;\;\nu\|h_\tau'\|_{L^\infty(t_1,t_2)}<\epsilon, \;\;\nu(t_2-t_1)<\epsilon/2, \;\;\mbox{and}
\]
\[
\nu(x_2-x_1)\|h_\tau'\|_{L^\infty(t_1,t_2)}<\epsilon.
\]
Then (a) and (c) are also fulfilled.
\end{proof}

\section{Proof of density lemma}\label{sec:density-proof}
As the final task of the paper, we provide a long proof of the density lemma, Lemma \ref{lem:density-fact}.

\subsection*{Proof of Lemma \ref{lem:density-fact}}
 Let $E=\Omega_T^2.$ Fix a  $\delta>0$, and let $w=(u,v)\in\mathcal{A}$;  so
\begin{equation}\label{property-w-in-density}
\left\{
\begin{array}{l}
  \mbox{$w\in C^{2,1}(\bar\Omega_T)\times C^{3,1}(\bar\Omega_T)$,} \\
  \mbox{$w=w^*$ in $\Omega_T\setminus\bar\Omega^w_T$ for some regular open set $\Omega^w_T\subset\subset\Omega^2_T$,} \\
  \mbox{$\nabla w(x,t)\in U((x,t);u)$ $\forall(x,t)\in\Omega^2_T$,} \\
  \|u-u^*\|_\infty <\epsilon/2,\;\; \|u_t-u^*_t\|_\infty <\epsilon/2,\\
  \mbox{$\mathcal{P}_u=\mathcal{P}_{u^*}$ in $\Omega_T\setminus\bar\Omega^w_T$, and $|\Gamma^E_w-\Gamma^E_{w^*}|<{\epsilon}/{2}$.}
\end{array}\right.
\end{equation}
Let $\eta>0$. Our goal is to construct a function $w_\eta=(u_\eta,v_\eta)\in\mathcal{A}_\delta$ with $\|w_\eta-w\|_\infty<\eta$, that is, a function $w_\eta\in C^{2,1}(\bar\Omega_T)\times C^{3,1}(\bar\Omega_T)$ satisfying
\begin{equation}\label{goal-density}
\left\{
\begin{array}{l}
  \mbox{$w_{\eta}=w^*$ in $\Omega_T\setminus\bar\Omega^{w_{\eta}}_T$ for some regular open set $\Omega^{w_{\eta}}_T\subset\subset\Omega^2_T$,} \\
  \mbox{$\nabla w_{\eta}(x,t)\in U((x,t);u_\eta)$ $\forall(x,t)\in\Omega^2_T$,} \\
  \|u_\eta-u^*\|_\infty <\epsilon/2,\;\; \|(u_\eta)_t-u^*_t\|_\infty <\epsilon/2,\\
  \mbox{$\mathcal{P}_{u_{\eta}}=\mathcal{P}_{u^*}$ in $\Omega_T\setminus\bar\Omega^{w_{\eta}}_T$,  $|\Gamma^E_{w_\eta}-\Gamma^E_{w^*}|<\epsilon/2$,} \\
  \int_{\Omega^2_T}\mathrm{dist} \big(\nabla w_{\eta}(x,t),K((x,t);u_{\eta})\big)\,dxdt \le\delta|\Omega^2_T|,\;\mbox{and} \\
  \|w_\eta-w\|_\infty<\eta.
\end{array}\right.
\end{equation}

The construction is so long that we divide it  into several steps.

\subsection*{\underline{\textbf{Step 1}}} First, we choose a regular open set $G\subset\subset \Omega^2_T\setminus\partial\Omega^w_T$ so that
\begin{equation}\label{density-modifiable-part}
\int_{\Omega^2_T\setminus\bar{G}}\mathrm{dist} \big(\nabla w(x,t),K((x,t);u)\big)\,dxdt \le\frac{\delta}{k}|\Omega^2_T|,
\end{equation}
where $k\in\N$ is  to be specified later.

By (\ref{property-w-in-density}), we have
\[
(u_x,v_t-bu-\mathcal{P}_u-F) \in U', \quad |u_t|<m^*\;\;\mbox{on $\bar{G}$};
\]
so
\begin{equation}\label{density-distance}
\begin{split} &d':=\min_{\bar{G}}\mathrm{dist}\big( (u_x,v_t-bu-\mathcal{P}_u-F),\partial U' \big)>0,\\
&m':=m^*-\max_{\bar{G}}|u_t|>0.\end{split}
\end{equation}
By (\ref{property-w-in-density}), we also have
\[
d'':=(\epsilon/2-|\Gamma^E_w-\Gamma^E_{w^*}|)/2>0.
\]
By the uniform continuity of $s^\pm_r$ for $r\in[\sigma(s_2),\sigma(s_1)]$, we can choose a $\kappa>0$ such that
\begin{equation}\label{density-kappa-condition}
\mbox{$|s^\pm_{a}-s^\pm_{b}|\le d''/l$ whenever $a,b \in[\sigma(s_2),\sigma(s_1)]$ and $|a-b|\le\kappa$},
\end{equation}
where $l\in\N$ will be chosen later.

Next, we choose finitely many disjoint open squares $Q_1,\cdots,Q_N\subset G$, parallel to the axes, such that
\begin{equation}\label{density-squre-part}
\int_{G\setminus(\cup_{i=1}^N\bar{Q}_i)}\mathrm{dist} \big(\nabla w(x,t),K((x,t);u)\big)\,dxdt \le\frac{\delta}{k}|\Omega^2_T|.
\end{equation}
Dividing these squares $Q_1,\cdots,Q_N\subset G$ into disjoint open sub-squares if necessary, we can assume that
\begin{equation}\label{density-distance-square}
\begin{split}
&\big| \big(  u_x(x_1, t_1) ,  v_t(x_1,t_1)  -b(x_1,t_1)u(x_1,t_1)  -\mathcal{P}_u(x_1,t_1)-F(x_1,t_1)\big) \\
& - \big(u_x(x_2,t_2) , v_t(x_2,t_2)-b(x_2,t_2)u(x_2,t_2)  -\mathcal{P}_u(x_2,t_2) -F(x_2,t_2)\big) \big| \\
& \;\;\; \le \min\Big\{\frac{\delta}{4k},\frac{d'}{4},\frac{\kappa}{2} \Big\}
\end{split}
\end{equation}
and
\begin{equation}\label{density-distance-square-1}
|u_t(x_1,t_1)-u_t(x_2,t_2)|\le\frac{m'}{4}
\end{equation}
whenever $(x_1,t_1),(x_2,t_2)\in\bar{Q}_i$ and $i\in\{1,\cdots,N\}$.

\subsection*{\underline{\textbf{Step 2}}}  Fix an index $i\in\mathcal{I}:=\{1,\cdots,N\}$, and let us denote by $(x_i,t_i)$ the center of the square $Q_i$. We write
\[
%\begin{split}
(s_i,\gamma_i)  = \big(u_x(x_i,  t_i) ,  v_t(x_i,t_i)  -b(x_i,t_i)u(x_i,t_i)-\mathcal{P}_u(x_i,t_i) -F(x_i,t_i)\big),
\]
\[
c_i  = u_t(x_i,t_i);%,\;\;u_i=u(x_i,t_i),\;\; q_i=v_t(x_i,t_i);
%\end{split}
\]
then $(s_i,\gamma_i)\in U'$ and $|c_i|\le m^*-m'.$

We now split the index set $\mathcal{I}$ into
\[
\mathcal{I}_1 := \{i\in\mathcal{I} \,|\, \mathrm{dist}((s_i,\gamma_i),K')\le\delta/k\},\;\;\mathcal{I}_2:=\mathcal{I}\setminus \mathcal{I}_1.
\]
Let $i\in\mathcal{I}_1$. Choose a point $\bar{s}_i\in [s^-_{r_1},s^-_{r_2}]\cup[s^+_{r_1},s^+_{r_2}]$ so that
\begin{equation}\label{density-point-chosen}
|(s_i,\gamma_i)-(\bar{s}_i,\sigma(\bar{s}_i))|\le\frac{\delta}{k}.
\end{equation}
Let $(x,t)\in Q_i$; then by (\ref{density-distance-square}) and (\ref{density-point-chosen}), we have
\[
\begin{split}
\bigg| &   \begin{pmatrix} u_x(x,t) & u_t(x,t) \\ v_x(x,t) & v_t(x,t) \end{pmatrix}  - \begin{pmatrix} \bar{s}_i & u_t(x,t) \\ u(x,t) & b(x,t)u(x,t)+\mathcal{P}_u(x,t)+F(x,t)+\sigma(\bar{s}_i) \end{pmatrix}\bigg| \\
& \quad\quad =\big|\big(u_x(x,t)-\bar{s}_i , v_t(x,t)-b(x,t)u(x,t)-\mathcal{P}_u(x,t)-F(x,t)-\sigma(\bar{s}_i)\big)\big| \\
& \quad\quad \le\big|\big(u_x(x,t) , v_t(x,t)-b(x,t)u(x,t)-\mathcal{P}_u(x,t)-F(x,t)\big)-(s_i,\gamma_i)\big| \\
& \quad\quad\quad\, +|(s_i,\gamma_i)-(\bar{s}_i,\sigma(\bar{s}_i))|\le\frac{5\delta}{4k},
\end{split}
\]
where
\[
\begin{pmatrix} \bar{s}_i & u_t(x,t) \\ u(x,t) & b(x,t)u(x,t)+\mathcal{P}_u(x,t)+F(x,t)+\sigma(\bar{s}_i) \end{pmatrix} \in K((x,t);u).
\]
So
\[
\int_{Q_i} \mathrm{dist}\big(\nabla w(x,t),K((x,t);u)\big)\,dxdt \le \frac{5\delta}{4k}|Q_i|.
\]
We thus have
\begin{equation}\label{density-close-integral}
\int_{\cup_{i\in\mathcal{I}_1}Q_i} \mathrm{dist}\big(\nabla w(x,t),K((x,t);u)\big)\,dxdt \le \frac{5\delta}{4k}|\Omega^2_T|.
\end{equation}

\subsection*{\underline{\textbf{Step 3}}} Fix an index $i\in\mathcal{I}_2$. Since
 $(s_i,\gamma_i)\in U'$ and $\mathrm{dist}((s_i,\gamma_i),K')>\delta/k$, we can choose two numbers $\lambda_{i,1}>0$ and $\lambda_{i,2}>0$ so that
\begin{equation}\label{density-two-points-move}
\left\{
\begin{array}{l}
  (s_i-\lambda_{i,1},\gamma_i),(s_i+\lambda_{i,2},\gamma_i)\in U', \\
  \mathrm{dist}((s_i-\lambda_{i,1},\gamma_i),K')= \mathrm{dist}((s_i+\lambda_{i,2},\gamma_i),K')=\frac{\delta}{k}, \\
  \mathrm{dist}((s_i+\mu,\gamma_i),K')>\frac{\delta}{k}\;\;\forall \mu\in(-\lambda_{i,1},\lambda_{i,2}).
\end{array}\right.
\end{equation}
%Let us write $d_i=\mathrm{dist}([(s_i-\lambda_{i,1},\gamma_i),(s_i+\lambda_{i,2},\gamma_i), \partial U'])>0$.
For a given $\epsilon_i>0$ to be specified later, we can apply Lemma \ref{lem:rank-1-smooth-approx} to obtain a function $\omega_i=(\varphi_i,\psi_i)\in C^\infty_c(Q_i;\R^2)$ such that
\begin{itemize}
\item[(a)] $\|\omega_i\|_{L^\infty(Q_i)}<\epsilon_i$, $\|(\varphi_i)_t\|_{L^\infty(Q_i)}<\epsilon_i$, $\|(\psi_i)_t\|_{L^\infty(Q_i)}<\epsilon_i$,
\item[(b)] $-\lambda_{i,1}\le(\varphi_i)_x\le\lambda_{i,2}$ in $Q_i$,
\item[(c)] $\left\{\begin{array}{l}
              \big| |\{  (x,t)\in Q_i\, |\, (\varphi_i)_x(x,t) =-\lambda_{i,1} \} |-\frac{\lambda_{i,2}}{\lambda_{i,1}+\lambda_{i,2}}|Q_i|\big| < \epsilon_i,\\
              \big| |\{  (x,t)\in Q_i\, |\, (\varphi_i)_x(x,t) =\lambda_{i,2} \} |-\frac{\lambda_{i,1}}{\lambda_{i,1}+\lambda_{i,2}}|Q_i|\big| < \epsilon_i,
            \end{array}\right.
$
\item[(d)] $(\psi_i)_x=\varphi_i$ in $Q_i$, and
\item[(e)] $\int_{x_{i,1}}^{x_{i,2}}\varphi_i(x,t)\,dx=0$  for all $t_{i,1}<t<t_{i,2}$,
\end{itemize}
where $Q_i=(x_{i,1},x_{i,2})\times(t_{i,1},t_{i,2})$. We then define
\begin{equation}\label{density-definition-w-eta}
w_\eta=(u_\eta,v_\eta)=w+\sum_{i\in\mathcal{I}_2}\omega_i\chi_{Q_i}\;\;\mbox{in $\Omega_T$}.
\end{equation}

\subsection*{\underline{\textbf{Step 4}}} To finish the proof, we now show that upon choosing suitable numbers $l,k\in\N$ and $\epsilon_i>0$ $(i\in\mathcal{I}_2)$, the function $w_\eta=(u_\eta,v_\eta)$ defined in (\ref{density-definition-w-eta}) satisfies the required properties in (\ref{goal-density}). As this step consists of many arguments to verify, we separate it into several substeps.

\subsubsection*{\underline{\textbf{Substep 4-1}}} We begin with relatively easy parts to show.

From (\ref{property-w-in-density}) and (\ref{density-definition-w-eta}), we have
\begin{equation}\label{goal-1}
w_\eta\in C^{2,1}(\bar\Omega_T)\times C^{3,1}(\bar\Omega_T).
\end{equation}

Set $\Omega^{w_\eta}_T=G\cup\Omega^w_T$; then $\Omega^{w_\eta}_T\subset\subset\Omega^2_T$ is a regular open set. From (\ref{property-w-in-density}) and (\ref{density-definition-w-eta}), we also have
\begin{equation}\label{goal-2}
w_\eta=w=w^*\;\;\mbox{in $\Omega_T\setminus\bar\Omega^{w_\eta}_T$}.
\end{equation}

Next, we choose
\begin{equation}\label{density-ep-choice-2}
\mbox{$\epsilon_i<\min\Big\{\eta,\; \frac{\epsilon}{2} -\|u-u^*\|_\infty,\, \frac{\epsilon}{2}-\|u_t-u^*_t\|_\infty\Big \}$\; for all $i\in\mathcal{I}_2$;}
\end{equation}
then  from (a) and (\ref{density-definition-w-eta}) we have
\begin{equation}\label{goal-3}
\begin{cases}
\|w_\eta-w\|_\infty<\eta,\\
\|u_\eta-u^*\|_\infty<{\epsilon}/{2},\\
\|(u_\eta)_t-u^*_t\|_\infty<{\epsilon}/{2}.
\end{cases}
\end{equation}

Note that by (\ref{coefficient-condition}), (e), (\ref{property-w-in-density}) and (\ref{density-definition-w-eta}), for all $(x,t)\in\Omega_T\setminus\bar{\Omega}^{w_\eta}_T$,
\[
\begin{split}
\mathcal{P}_{u_\eta}(x,t) & =d(t)\int_0^x u_\eta(y,t)\,dy\\
& = d(t)\int_0^x u(y,t)\,dy + d(t)\sum_{i\in\mathcal{I}_2}\int_0^x \varphi_i(y,t)\chi_{Q_i}(y,t)\,dy\\
& =\mathcal{P}_u(x,t) =\mathcal{P}_{u^*}(x,t);
\end{split}
\]
that is,
\begin{equation}\label{goal-4}
\mathcal{P}_{u_\eta}=\mathcal{P}_{u^*}\;\;\mbox{in $\Omega_T\setminus\bar{\Omega}^{w_\eta}_T$}.
\end{equation}

\subsubsection*{\underline{\textbf{Substep 4-2}}} In this substep, we show that
\begin{equation}\label{goal-5}
\nabla{w}_\eta(x,t)\in U((x,t);u_\eta)\;\;\forall (x,t)\in\Omega^2_T.
\end{equation}

Note from (\ref{density-distance-square-1}), (\ref{density-definition-w-eta}) and (a) that
\[
\begin{split}
|(u_\eta)_t(x,t)| & =|u_t(x,t)+(\varphi_i)_t(x,t)|\le |u_t(x,t)-c_i|+|c_i|+|(\varphi_i)_t(x,t)| \\
& \le \frac{m'}{4}+m^*-m'+ \epsilon_i< m^*-\frac{m'}{2}\;\;\forall (x,t)\in Q_i,\,\forall i\in\mathcal{I}_2,
\end{split}
\]
where we let
\begin{equation}\label{density-ep-choice-3}
\mbox{$\epsilon_i< m'/4$\; for each $i\in\mathcal{I}_2$.}
\end{equation}
This  implies that
\begin{equation}\label{density-U-inclusion-1}
|(u_\eta)_t(x,t)|< m^*\;\;\forall (x,t)\in\Omega^2_T.
\end{equation}

From (\ref{property-w-in-density}), we have $v_x=u$ in $\Omega^2_T$. Thus by (\ref{density-definition-w-eta}) and (d), we have
\[
\begin{split}
(v_\eta)_x (x,t) & =v_x(x,t) + (\psi_i)_x(x,t) \\
& =u(x,t) + \varphi_i(x,t)=u_\eta(x,t)\;\;\forall (x,t)\in Q_i,
\end{split}
\]
where $i\in\mathcal{I}_2$; that is,
\begin{equation}\label{density-U-inclusion-2}
(v_\eta)_x (x,t)  =u_\eta(x,t)\;\;\forall (x,t)\in \Omega^2_T.
\end{equation}

Let $i\in\mathcal{I}_2$.
From (\ref{density-distance}) and (\ref{density-two-points-move}), we see that
\begin{equation}\label{density-distance-freezed-bar}
\mathrm{dist}\big( [(s_i-\lambda_{i,1},\gamma_i),(s_i+\lambda_{i,2},\gamma_i)] , \partial U' \big)\ge \min\Big\{\frac{\delta}{k},d'\Big\}.
\end{equation}
Let $(x,t)\in Q_i.$ Then by (a), (e) and (\ref{density-distance-square}),
\[
\begin{split}
\big|\big((u_\eta)_x  & (x,t), (v_\eta)_t(x,t)- b(x,t)u_\eta(x,t)-\mathcal{P}_{u_\eta}(x,t) -F(x,t)\big) \\
& - \big(s_i+(\varphi_i)_x(x,t),\gamma_i\big)\big| \\
\le &   \big|\big(u_x(x,t), v_t(x,t)- b(x,t)u(x,t)-\mathcal{P}_{u}(x,t)-F(x,t)\big) - (s_i,\gamma_i)\big| \\
& + \bigg|(\psi_i)_t(x,t) -b(x,t)\varphi_i(x,t) -d(t)\int_{x_{i,1}}^x\varphi_i(y,t)\,dy \bigg| \\
\le & \min\Big\{\frac{\delta}{4k},\frac{d'}{4}\Big\} + \epsilon_i(1 + \|b\|_\infty +\|d\|_{L^\infty(0,T)}) <\min\Big\{\frac{\delta}{2k},\frac{d'}{2}\Big\},
\end{split}
\]
where we let
\begin{equation}\label{density-ep-choice-1}
\epsilon_i< (1 + \|b\|_\infty +\|d\|_{L^\infty(0,T)})^{-1} \min\Big\{\frac{\delta}{4k},\frac{d'}{4}\Big\}.
\end{equation}
By (b), we have $-\lambda_{i,1}\le(\varphi_i)_x(x,t)\le\lambda_{i,2}$, and so it follows from  (\ref{density-distance-freezed-bar}) and the previous inequality that
\[
\big((u_\eta)_x(x,t), (v_\eta)_t(x,t)- b(x,t)u_\eta(x,t)-\mathcal{P}_{u_\eta}(x,t) -F(x,t)\big)\in U'.
\]
Adopting (\ref{density-ep-choice-1}) for all $i\in\mathcal{I}_2$, this inclusion holds for all $(x,t)\in\Omega^2_T.$

This inclusion together with (\ref{density-U-inclusion-1}) and (\ref{density-U-inclusion-2}) implies inclusion (\ref{goal-5}).

\subsubsection*{\underline{\textbf{Substep 4-3}}} Here, we prove that
\begin{equation}\label{goal-7}
|\Gamma^E_{w_\eta}-\Gamma^E_{w^*}|<{\epsilon}/{2}.
\end{equation}

Observe from (\ref{fractional-distance}) that we have
\[
\begin{split}
 \Gamma^E_{w_\eta}-\Gamma^E_w = &   \frac{1}{|\Omega^2_T|} \int_{\Omega^2_T} (Z_{w_\eta}(x,t)-Z_{w}(x,t)) \,dxdt \\
= &  \frac{1}{|\Omega^2_T|} \sum_{i\in\mathcal{I}_2} \int_{Q_i}  (Z_{w+\omega_i}(x,t)-Z_{w}(x,t))\,dxdt \\
= & \frac{1}{|\Omega^2_T|} \sum_{i\in\mathcal{I}_2} \bigg(\int_{Q_{i,1}}  + \int_{Q^+_{i,2}}  + \int_{Q^-_{i,2}} \bigg)(Z_{w+\omega_i}-Z_{w}),
\end{split}
\]
where
\[
\begin{split}
Q_{i,1} & :=\{(x,t)\in Q_i\,|\, (\varphi_i)_x(x,t)\not\in\{-\lambda_{i,1},\lambda_{i,2}\}  \}, \\
Q^+_{i,2} & :=\{(x,t)\in Q_i\,|\, (\varphi_i)_x(x,t)=\lambda_{i,2}  \}, \\
Q^-_{i,2} & :=\{(x,t)\in Q_i\,|\, (\varphi_i)_x(x,t)=-\lambda_{i,1}  \}.
\end{split}
\]
For the second equality above, we have used the facts  from (e) that $\mathcal{P}_{u_\eta}(x,t)=\mathcal{P}_{u}(x,t)$ for all $(x,t)\not\in\cup_{i\in\mathcal{I}_2}Q_i$ and that $\mathcal{P}_{u_\eta}(x,t)=\mathcal{P}_{u+\varphi_i}(x,t)$ for all $(x,t)\in Q_i$ and $i\in\mathcal{I}_2;$ here, $\omega_i$ is defined to be zero outside its compact support.

The trivial part to estimate here follows from (c) as
\begin{equation}\label{goal-7-estimate-1}
\bigg|\int_{Q_{i,1}} (Z_{w+\omega_i}-Z_{w})  \bigg|\le 2|Q_{i,1}|<4\epsilon_i\;\;(i\in\mathcal{I}_2).
\end{equation}

Next, let $i\in\mathcal{I}_2$ and $(x,t)\in Q^-_{i,2}$. Then from (\ref{density-distance-square}), we have at the point $(x,t)$ that
\begin{equation}\label{density-numerator-small}
\begin{split}
|(u+\varphi_i)_x- & S^-_{w+\omega_i}|=\big|(u+\varphi_i)_x - s^-_{(v+\psi_i)_t -b (u +\varphi_i ) -\mathcal{P}_{u+\varphi_i}  -F}\big| \\
= & \big|u_x  -\lambda_{i,1}-s^-_{(v+\psi_i)_t -b (u +\varphi_i ) -\mathcal{P}_{u+\varphi_i}  -F}\big| \\
\le & |u_x -s_i| + |s_i-\lambda_{i,1}-s^-_{\gamma_i}| \\
& + \big|s^-_{\gamma_i} -s^-_{(v+\psi_i)_t -b (u +\varphi_i ) -\mathcal{P}_{u+\varphi_i}  -F}\big| \\
\le & \frac{\delta}{4k}+|s_i-\lambda_{i,1}-s^-_{\gamma_i}|+ \big|s^-_{\gamma_i} -s^-_{(v+\psi_i)_t -b (u +\varphi_i ) -\mathcal{P}_{u+\varphi_i}  -F}\big|.
\end{split}
\end{equation}
By (\ref{density-two-points-move}), we have $|(s_i-\lambda_{i,1},\gamma_i)-(\bar{s},\sigma(\bar{s}))|=\delta/k$ for some  $\bar{s}\in[s^-_{\gamma_i},s_i-\lambda_{i,1}]$. Thus, we have from (\ref{density-kappa-condition}) that
\begin{equation}\label{density-sophistication-1}
|s_i-\lambda_{i,1}-s^-_{\gamma_i}|\le |s_i-\lambda_{i,1}-\bar{s}|+|\bar{s}-s^-_{\gamma_i}| \le\frac{\delta}{k}+\frac{d''}{l},
\end{equation}
where we let $k\in\N$ satisfy
\begin{equation}\label{density-k-choice-2}
\frac{\delta}{k}\le\kappa.
\end{equation}
Likewise, we also have
\begin{equation}\label{density-sophistication-3}
|s_i+\lambda_{i,2}-s^+_{\gamma_i}| \le\frac{\delta}{k}+\frac{d''}{l}.
\end{equation}
Also, from (\ref{density-distance-square}) and (a), we have
\[
\begin{split}
|\gamma_i-( & (v+\psi_i)_t -b (u +\varphi_i ) -\mathcal{P}_{u+\varphi_i}  -F )| \\
& \le |\gamma_i-( v_t  -b u   -\mathcal{P}_{u} -F )| \\
& \;\;\;\;\; +|(\psi_i)_t |+|b||\varphi_i| + |d(t)|\bigg|\int_0^x\varphi_i(y,t)\,dy\bigg| \\
& \le \frac{\kappa}{2}+\epsilon_i\big (1+\|b\|_\infty+ \|d\|_{L^\infty(0,T)}\big )\le\kappa,
\end{split}
\]
where we let
\begin{equation}\label{density-ep-choice-5}
\epsilon_i\le 2^{-1}(1+\|b\|_\infty+ \|d\|_{L^\infty(0,T)})^{-1}\kappa.
\end{equation}
With the help of (\ref{density-kappa-condition}), this implies that 
\begin{equation}\label{density-sophistication-2}
|s^-_{\gamma_i}-S^-_{w+\omega_i}|=\big|s^-_{\gamma_i} -s^-_{(v+\psi_i)_t -b (u +\varphi_i ) -\mathcal{P}_{u+\varphi_i}  -F}\big|\le\frac{d''}{l}.
\end{equation}
Thus (\ref{density-numerator-small}), (\ref{density-sophistication-1}) and (\ref{density-sophistication-2}) yield that 
\[
\big| (u+\varphi_i)_x -s^-_{(v+\psi_i)_t -b (u +\varphi_i ) -\mathcal{P}_{u+\varphi_i}  -F} \big|\le \frac{5\delta}{4k}+\frac{2d''}{l}\le\frac{3d''}{l}
\]
if we choose $k\in\N$ so large that
\begin{equation}\label{density-k-choice-1}
\frac{5\delta}{4k}\le\frac{d''}{l}.
\end{equation}
Since $S^+_{w+\omega_i}-S^-_{w+\omega_i}\ge s^+_{r_1}-s^-_{r_2}$, we thus have
\[
Z_{w+\omega_i}(x,t) \le \frac{3d''}{l(s^+_{r_1}-s^-_{r_2})}\quad \forall\, (x,t)\in  Q_{i,2}^-.
\]
Hence
\begin{equation}\label{goal-7-estimate-2}
\bigg|\int_{Q^-_{i,2}} Z_{w+\omega_i}\bigg| \le \frac{3d''}{l(s^+_{r_1}-s^-_{r_2})}|Q^-_{i,2}|.
\end{equation}

Let $i\in\mathcal{I}_2$ and $Q_{i,2}:=Q_{i,2}^+\cup Q_{i,2}^-$. We now estimate the quantity
\[
\bigg| \int_{Q_{i,2}^+} Z_{w+\omega_i} - \int_{Q_{i,2}} Z_{w}\bigg|.
\]

First, we consider
\[
\begin{split}
\bigg|\int_{Q_{i,2}^+} &  Z_{w+\omega_i}(x,t) \,dxdt - \frac{\lambda_{i,1}}{\lambda_{i,1}+\lambda_{i,2}}|Q_i|\bigg| \\
\le & \bigg|\int_{Q_{i,2}^+} \bigg( \frac{u_x+(\varphi_i)_x-S^-_{w+\omega_i}}{S^+_{w+\omega_i}-S^-_{w+\omega_i}} - \frac{s_i+\lambda_{i,2}-s^-_{\gamma_i}}{S^+_{w+\omega_i}-S^-_{w+\omega_i}}\bigg) \,dxdt\bigg|\\
& + \bigg|\int_{Q_{i,2}^+} \bigg( \frac{s_i+\lambda_{i,2}-s^-_{\gamma_i}}{S^+_{w+\omega_i}-S^-_{w+\omega_i}} - \frac{s_i+\lambda_{i,2}-s^-_{\gamma_i}}{s^+_{\gamma_i}-s^-_{\gamma_i}}\bigg) \,dxdt\bigg|\\
& + \bigg| \frac{s_i+\lambda_{i,2}-s^-_{\gamma_i}}{s^+_{\gamma_i}-s^-_{\gamma_i}}|Q^+_{i,2}|- \frac{\lambda_{i,1}}{\lambda_{i,1}+\lambda_{i,2}}|Q_i| \bigg|=:I_1+I_2+I_3.
\end{split}
\]
Here, let $\epsilon_i>0$ satisfy (\ref{density-ep-choice-5}); then as in (\ref{density-sophistication-2}), we have
\begin{equation}\label{density-extra-1}
|S^\pm_{w+\omega_i}-s^\pm_{\gamma_i}|\le\frac{d''}{l}\;\;\mbox{in $Q^+_{i,2}$}.
\end{equation}
Using this, (\ref{density-distance-square}) and the fact that $(\varphi_i)_x=\lambda_{i,2}$ in $Q^+_{i,2}$, we deduce that
\[
I_1\le (s^+_{r_1}-s^-_{r_2})^{-1}\Big(\frac{\delta}{4k}+\frac{d''}{l}\Big)|Q^+_{i,2}|\le \frac{6d''}{5l(s^+_{r_1}-s^-_{r_2})}|Q^+_{i,2}|,
\]
where we let $k\in\N$ fulfill (\ref{density-k-choice-1}). Having the common denominator in the integrand, we  have from (\ref{density-extra-1}) that
\[
I_2\le (s_i+\lambda_{i,2}-s^-_{\gamma_i})(s^+_{r_1}-s^-_{r_2})^{-2}\frac{2d''}{l} |Q^+_{i,2}| \le \frac{2d''(s^+_{r_2}-s^-_{r_1})}{l(s^+_{r_1}-s^-_{r_2})^{2}}|Q^+_{i,2}|.
\]
Note from (\ref{density-sophistication-3}) and (c) that
\[
\begin{split}
I_3 & \le \bigg| \frac{s_i+\lambda_{i,2}-s^-_{\gamma_i}}{s^+_{\gamma_i}-s^-_{\gamma_i}}|Q^+_{i,2}|- |Q^+_{i,2}|\bigg|+\bigg||Q^+_{i,2}|- \frac{\lambda_{i,1}}{\lambda_{i,1}+\lambda_{i,2}}|Q_i| \bigg| \\
& \le \Big(\frac{\delta}{k}+\frac{d''}{l}\Big) (s^+_{\gamma_i}-s^-_{\gamma_i})^{-1}|Q^+_{i,2}|+\epsilon_i \le \frac{9d''}{5l(s^+_{r_1}-s^-_{r_2})}|Q^+_{i,2}|+\epsilon_i,
\end{split}
\]
where we let $k\in\N$ also satisfy (\ref{density-k-choice-2}). Combining the estimates on $I_1,$ $I_2$ and $I_3$, we obtain
\begin{equation}\label{density-intermediate-est}
\begin{split}
\bigg|\int_{Q_{i,2}^+}  &Z_{w+\omega_i} (x,t)\,dxdt -  \frac{\lambda_{i,1}}{\lambda_{i,1}+\lambda_{i,2}}|Q_i|\bigg|  \\
& \quad\quad\le \bigg(\frac{3}{s^+_{r_1}-s^-_{r_2}} + \frac{2(s^+_{r_2}-s^-_{r_1})}{(s^+_{r_1}-s^-_{r_2})^{2}}\bigg)\frac{d''}{l} |Q^+_{i,2}|+\epsilon_i.
\end{split}
\end{equation}

Second, we handle
\[
\begin{split}
\bigg|\int_{Q_{i,2}} & Z_w(x,t)\,dxdt - \frac{\lambda_{i,1}}{\lambda_{i,1}+\lambda_{i,2}}|Q_i|\bigg| \\
\le &  \bigg|\int_{Q_{i,2}}  \bigg (\frac{u_x-S^-_{w}}{S^+_{w}-S^-_{w}} - \frac{\lambda_{i,1}}{S^+_{w}-S^-_{w}}\bigg ) \,dxdt\bigg| \\
& + \bigg|\int_{Q_{i,2}}   \frac{\lambda_{i,1}}{S^+_{w}-S^-_{w}}\,dxdt -\frac{\lambda_{i,1}}{s^+_{\gamma_i}-s^-_{\gamma_i}}|Q_{i,2}|\bigg|\\
& + \bigg|\frac{\lambda_{i,1}}{s^+_{\gamma_i}-s^-_{\gamma_i}}|Q_{i,2}| - \frac{\lambda_{i,1}}{\lambda_{i,1}+\lambda_{i,2}}|Q_i|\bigg|=:J_1+J_2+J_3.
\end{split}
\]
To estimate $J_1$, we first note from (\ref{density-kappa-condition}), (\ref{density-distance-square}) and (\ref{density-sophistication-1}) that for $(x,t)\in Q_{i,2}$,
\[
\begin{split}
|u_x(x,t)-S^-_{w}{(x,t)}-\lambda_{i,1}| & \le |u_x(x,t)-s_i|+|s^-_{\gamma_i}-S^-_{w}{(x,t)}| +| s_i-\lambda_{i,1} - s^-_{\gamma_i}|\\
& \le \frac{\delta}{4k}+\frac{d''}{l}+\Big(\frac{\delta}{k}+\frac{d''}{l}\Big)\le\ \frac{3d''}{l},
\end{split}
\]
where we let $k\in\N$ satisfy (\ref{density-k-choice-2}) and (\ref{density-k-choice-1}). From this, we get
\[
J_1\le \frac{3d''}{l(s^+_{r_1}-s^-_{r_2})}|Q_{i,2}|.
\]
From (\ref{density-kappa-condition}) and (\ref{density-distance-square}), we have
\[
J_2\le \frac{2\lambda_{i,1}d''}{l(s^+_{r_1}-s^-_{r_2})^2}|Q_{i,2}|\le \frac{2d''(s^+_{r_2}-s^-_{r_1})}{l(s^+_{r_1}-s^-_{r_2})^2}|Q_{i,2}|.
\]
To estimate $J_3$, we observe from (\ref{density-sophistication-1}) and (\ref{density-sophistication-3}) that
\[
\begin{split}
|\lambda_{i,1}+\lambda_{i,2}-(s^+_{\gamma_i}-s^-_{\gamma_i})| & =|s_i+\lambda_{i,2} -s^+_{\gamma_i}-(s_i-\lambda_{i,1} -s^-_{\gamma_i})| \\
&\le \frac{2\delta}{k}+ \frac{2d''}{l}\le \frac{18d''}{5l}.
\end{split}
\]
Note from this and (c) that
\[
\begin{split}
J_3 & \le \bigg|\frac{\lambda_{i,1}}{s^+_{\gamma_i}-s^-_{\gamma_i}}|Q_{i,2}| - \frac{\lambda_{i,1}}{\lambda_{i,1}+\lambda_{i,2}}|Q_{i,2}|\bigg| + \bigg|\frac{\lambda_{i,1}}{\lambda_{i,1}+\lambda_{i,2}}|Q_{i,2}| - \frac{\lambda_{i,1}}{\lambda_{i,1}+\lambda_{i,2}}|Q_i|\bigg| \\
& \le \frac{18d''(s^+_{r_2}-s^-_{r_1})}{5l(s^+_{r_1}-s^-_{r_2})^2}|Q_{i,2}|+2\epsilon_i.
\end{split}
\]
By the estimates on $J_1$, $J_2$ and $J_3$, we now have
\[
\begin{split}
\bigg|\int_{Q_{i,2}}  Z_w(x,t) \,dxdt &- \frac{\lambda_{i,1}}{\lambda_{i,1}+\lambda_{i,2}}|Q_i|\bigg| \\
& \le \bigg(\frac{3}{s^+_{r_1}-s^-_{r_2}}+ \frac{28(s^+_{r_2}-s^-_{r_1})}{5(s^+_{r_1}-s^-_{r_2})^2}\bigg)\frac{d''}{l}|Q_{i,2}|+2\epsilon_i.
\end{split}
\]

Combining this estimate with (\ref{density-intermediate-est}), we have
\begin{equation}\label{goal-7-estimate-3}
\begin{split}
\bigg| \int_{Q_{i,2}^+} &Z_{w+\omega_i}  - \int_{Q_{i,2}}Z_{w}\bigg|  \\
& \quad\quad\quad\le \bigg(\frac{6}{s^+_{r_1}-s^-_{r_2}}+ \frac{38(s^+_{r_2}-s^-_{r_1})}{5(s^+_{r_1}-s^-_{r_2})^2}\bigg)\frac{d''}{l} |Q_{i,2}|+3\epsilon_i.
\end{split}
\end{equation}

Thanks to the estimates (\ref{goal-7-estimate-1}), (\ref{goal-7-estimate-2}) and (\ref{goal-7-estimate-3}), it follows that for all $i\in\mathcal{I}_2,$
\[
\begin{split}
\bigg|\bigg ( \int_{Q_{i,1}}+&\int_{Q^+_{i,2}}+\int_{Q^-_{i,2}}\bigg) (Z_{w+\omega_i} -Z_w) \bigg| \\
& \le \bigg(\frac{9}{s^+_{r_1}-s^-_{r_2}}+ \frac{38(s^+_{r_2}-s^-_{r_1})}{5(s^+_{r_1}-s^-_{r_2})^2}\bigg)\frac{d''}{l}|Q_{i,2}|+7\epsilon_i\\
& \le  \frac{17d''(s^+_{r_2}-s^-_{r_1})}{l(s^+_{r_1}-s^-_{r_2})^2}|Q_{i}|+7\epsilon_i.
\end{split}
\]
Summing this over the indices $i\in\mathcal{I}_2,$ we then have
\[
|\Gamma^E_{w_\eta}-\Gamma^E_w|\le  \frac{17d''(s^+_{r_2}-s^-_{r_1})}{l(s^+_{r_1}-s^-_{r_2})^2} + \frac{7}{|\Omega^2_T|}\sum_{i\in\mathcal{I}_2}\epsilon_i< d''
\]
if $l\in\N$ is taken so large that
\begin{equation}\label{density-l-choice-1}
\frac{17(s^+_{r_2}-s^-_{r_1})}{l(s^+_{r_1}-s^-_{r_2})^2}<\frac{1}{2},
\end{equation}
$k\in\N$ satisfies (\ref{density-k-choice-2}) and (\ref{density-k-choice-1}), and the numbers $\epsilon_i>0$ $(i\in\mathcal{I}_2)$ fulfill (\ref{density-ep-choice-5}) and
\[
\epsilon_i<\frac{d''}{14N}|\Omega^2_T|.
\]
Under these choices of numbers $l,$ $k$ and $\epsilon_i$ $(i\in\mathcal{I}_2)$, we finally have from (\ref{property-w-in-density}) and the definition of $d''$ that
\[
\begin{split}
|\Gamma^E_{w_\eta}-\Gamma^E_{w^*}| & \le |\Gamma^E_{w_\eta}-\Gamma^E_{w}|+|\Gamma^E_{w}-\Gamma^E_{w^*}| \\
& <d'' +|\Gamma^E_{w}-\Gamma^E_{w^*}|=\frac{\epsilon}{4}+\frac{|\Gamma^E_{w}-\Gamma^E_{w^*}|}{2} <\frac{\epsilon}{2};
\end{split}
\]
hence, our goal (\ref{goal-7}) for this substep is indeed achieved.

\subsubsection*{\underline{\textbf{Substep 4-4}}} We now prove
\begin{equation}\label{goal-8}
\int_{\Omega^2_T} \mathrm{dist}\big(\nabla w_\eta(x,t) ,K((x,t);u_\eta)\big)\,dxdt\le \delta|\Omega_T^2|.
\end{equation}

Let $i\in\mathcal{I}_2$.
By (c), we have $|Q_{i,1}|<2\epsilon_i$. Let $(x,t)\in Q_{i,1}$ and $s\in[s^-_{r_1},s^-_{r_2}]\cup[s^+_{r_1},s^+_{r_2}]$; then by (\ref{goal-5}), we have, at the point $(x,t)$,
\[
\begin{split}
  \bigg|    \begin{pmatrix} (u_\eta)_x  & (u_\eta)_t  \\ (v_\eta)_x  & (v_\eta)_t  \end{pmatrix}  &- \begin{pmatrix} s & (u_\eta)_t \\ u_\eta & b u_\eta+\mathcal{P}_{u_\eta} +F +\sigma(s) \end{pmatrix}\bigg| \\
& =\big|\big((u_\eta)_x, (v_\eta)_t-bu_\eta-\mathcal{P}_{u_\eta} -F \big)-(s, \sigma(s))\big| \\
&  \le \mathrm{diam}(U').
\end{split}
\]
Thus, we have
\begin{equation*}%\label{density-close-integral}
\int_{Q_{i,1}} \mathrm{dist}\big(\nabla w_\eta(x,t),K((x,t);u_\eta)\big)\,dxdt \le 2\epsilon_i \mathrm{diam}(U')\le\frac{\delta}{Nk}|\Omega^2_T|,
\end{equation*}
where we let
\begin{equation}\label{density-ep-choice-4}
\epsilon_i\le (\mathrm{diam}(U'))^{-1}\frac{\delta}{2Nk}|\Omega^2_T|.
\end{equation}
Having this choice for all $i\in\mathcal{I}_2$, we get
\begin{equation}\label{density-Q-i-1}
\sum_{i\in\mathcal{I}_2}\int_{Q_{i,1}} \mathrm{dist}\big(\nabla w_\eta(x,t),K((x,t);u_\eta)\big)\,dxdt \le  \frac{\delta}{k}|\Omega^2_T|.
\end{equation}

Let $i\in\mathcal{I}_2$ and $(x,t)\in Q_{i,2}$; then $(\varphi_i)_x(x,t)\in\{-\lambda_{i,1},\lambda_{i,2}\}$. Suppose $(\varphi_i)_x(x,t)=-\lambda_{i,1}.$ By (\ref{density-two-points-move}), we can  choose a number $\bar{s}\in[s^-_{r_1},s^-_{r_2}]$ such that
\[
| (s_i-\lambda_{i,1},\gamma_i) - (\bar{s},\sigma(\bar{s})) |=\frac{\delta}{k}.
\]
Then it follows from (a), (e), (\ref{density-distance-square}) and the previous equality that  at the point $(x,t)$,
\[
\begin{split}
  \bigg|   \begin{pmatrix} (u_\eta)_x  & (u_\eta)_t \\ (v_\eta)_x  & (v_\eta)_t  \end{pmatrix}
& - \begin{pmatrix} \bar{s} & (u_\eta)_t \\ u_\eta  & b u_\eta +\mathcal{P}_{u_\eta} +F +\sigma(\bar{s}) \end{pmatrix}\bigg| \\
&  =\big|\big((u_\eta)_x, (v_\eta)_t-bu_\eta-\mathcal{P}_{u_\eta}-F\big)-(\bar{s}, \sigma(\bar{s}))\big| \\
& \le |(u_x, v_t-bu-\mathcal{P}_{u}-F) - (s_i,\gamma_i)| \\
&\quad\; + | (s_i-\lambda_{i,1},\gamma_i) - (\bar{s},\sigma(\bar{s})) | + |(\psi_i)_t|+|b||\varphi_i| \\
&  \quad\; +|d(t)|\bigg|\int_{x_{i,1}}^x\varphi_i(y,t)\,dy \bigg| \\
&  \le \frac{\delta}{4k} + \frac{\delta}{k} + \epsilon_i (1+\|b\|_\infty+\|d\|_{L^\infty(0,T)}) \le \frac{3\delta}{2k},
\end{split}
\]
where we let $\epsilon_i$ satisfy (\ref{density-ep-choice-1}). The same can be shown in the case that $(\varphi_i)_x(x,t)=\lambda_{i,2}$; we omit the details. Thus we get
\[
\int_{Q_{i,2}} \mathrm{dist}\big(\nabla w_\eta(x,t),K((x,t);u_\eta)\big)\,dxdt \le \frac{3\delta}{2k} |Q_{i,2}|,
\]
and so
\begin{equation}\label{density-Q-i-2}
\sum_{i\in\mathcal{I}_2}\int_{Q_{i,2}} \mathrm{dist}\big(\nabla w_\eta(x,t),K((x,t);u_\eta)\big)\,dxdt \le \frac{3\delta}{2k} |\Omega^2_T|.
\end{equation}

Gathering all of (\ref{density-modifiable-part}), (\ref{density-squre-part}), (\ref{density-close-integral}), (\ref{density-Q-i-1}) and  (\ref{density-Q-i-2}), we obtain
\begin{equation*}%\label{goal-6}
\begin{split}
\int_{\Omega^2_T} &\mathrm{dist}\big(\nabla w_\eta(x,t),K((x,t);u_\eta)\big)\,dxdt \\
= & \int_{\Omega^2_T\setminus\bar{G}} \mathrm{dist}\big(\nabla w(x,t),K((x,t);u)\big)\,dxdt \\
& + \int_{G\setminus(\cup_{i=1}^N\bar{Q}_i)} \mathrm{dist}\big(\nabla w(x,t),K((x,t);u)\big)\,dxdt \\
& + \int_{\cup_{i\in\mathcal{I}_1}Q_i} \mathrm{dist}\big(\nabla w(x,t),K((x,t);u)\big)\,dxdt \\
& + \sum_{i\in\mathcal{I}_2}\int_{Q_{i,1}} \mathrm{dist}\big(\nabla w_\eta(x,t),K((x,t);u_\eta)\big)\,dxdt \\
& + \sum_{i\in\mathcal{I}_2}\int_{Q_{i,2}} \mathrm{dist}\big(\nabla w_\eta(x,t),K((x,t);u_\eta)\big)\,dxdt \\
\le & \Big(\frac{\delta}{k}+\frac{\delta}{k} + \frac{5\delta}{4k}+\frac{\delta}{k} +\frac{3\delta}{2k}\Big)|\Omega^2_T| = \frac{23\delta}{4k}|\Omega^2_T|<\delta|\Omega^2_T|,
\end{split}
\end{equation*}
where we let $k\ge6$; hence, (\ref{goal-8}) is satisfied.

\subsubsection*{\underline{\textbf{Substep 4-5}}}
In conclusion, let us take an $l\in\N$ satisfying (\ref{density-l-choice-1}), choose a $k\in\N$ to fulfill $k\ge 6$ and (\ref{density-k-choice-1}) and numbers $\epsilon_i>0$ $(i\in\mathcal{I}_2)$ so that (\ref{density-ep-choice-2}), (\ref{density-ep-choice-3}), (\ref{density-ep-choice-1}), (\ref{density-ep-choice-5}) and (\ref{density-ep-choice-4}) are satisfied. Then we have (\ref{goal-1}), (\ref{goal-2}), (\ref{goal-3}), (\ref{goal-4}), (\ref{goal-5}), (\ref{goal-7}) and (\ref{goal-8}); that is, the function $w_\eta=(u_\eta,v_\eta)$ indeed fulfills all the desired properties in (\ref{goal-density}).

The proof of Lemma \ref{lem:density-fact}  is now complete.


\begin{thebibliography}{11}

%\bibitem{AR}
%N. Alikakos and R. Rostamian, {\em Gradient estimates for degenerate diffusion equations. I}, Math. Ann., {\bf 259} (1) (1982), 53--70.

%\bibitem{An}
%G. Andrews, {\em On the existence of solutions to the equation $u_{tt}=u_{xxt}+\sigma(u_x)_x$}, J. Differential Equations, {\bf 35} (1980), 200-231.

%\bibitem{AB}
%G. Andrews and J. M. Ball, {\em Asymptotic behaviour and changes of phase in
%one-dimensional nonlinear viscoelasticity}, J. Differential Equations, {\bf 44} (1982), 306--341.

%\bibitem{BJ}
%J.M. Ball and R.D. James,  {\em Fine phase mixtures as minimizers of energy},  {Arch. Rational Mech. Anal.}, {\bf 100} (1) (1987), 13--52.


%\bibitem{BF}
%G. Bellettini and G. Fusco, {\em The $\Gamma$-limit and the related gradient flow for singular perturbation functionals of Perona-Malik type},  {Trans. Amer. Math. Soc.}, {\bf 360} (9) (2008),  4929--4987.

\bibitem{BBDU}
G.~I.~Barenblatt, M.~Bertsch, R.~Dal Passo and M.~Ughi, {\em A degenerate pseudoparabolic regularization of a nonlinear forward-backward heat equation arising in the theory of heat and mass exchange in stably stratified turbulent shear flow}, SIAM J.\,Math.\,Anal., {\bf 24} (1993), 1414--1439.

\bibitem{BFG}
G. Bellettini, G. Fusco and N. Guglielmi, {\em A concept of solution and numerical experiments for forward-backward diffusion equations},  {Discrete Contin. Dyn. Syst.}, {\bf 16} (4) (2006),  783--842.

%\bibitem{BB}
%J. Bourgain and H. Brezis, {\em On the equation $\dv Y=f$ and application to control of phases}, {J. Amer. Math. Soc.}, {\bf 16} (2) (2002), 393--426.

%\bibitem{Br}
%H. Br\'ezis, ``Op\'erateurs maximaux monotones et semi-groupes de contractions dans les espaces de Hilbert," North-Holland Mathematics Studies, No. 5. Notas de Matem�tica (50). North-Holland Publishing Co., Amsterdam-London; American Elsevier Publishing Co., Inc., New York, 1973.



%\bibitem{BBT}
%A. Bruckner, J. Bruckner and B. Thomson, ``Real analysis," Prentice-Hall, 1996.

%\bibitem{CR}
%C. Carstensen and M. Rieger, {\em Young-measure approximations for elastodynamics with non-monotone stress-strain relations},  M2AN Math. Model. Numer. Anal., {\bf 38} (2004),  397--418.

%\bibitem{CZ}
%Y. Chen and K. Zhang, {\em Young measure solutions of the two-dimensional Perona-Malik equation in image processing},  Commun. Pure Appl. Anal., {\bf 5} (3) (2006),  615--635.

%\bibitem{CK}
%M. Chipot and D. Kinderlehrer,  {\em Equilibrium configurations of crystals}, Arch. Rational Mech. Anal., {\bf 103} (3) (1988), 237--277.

%\bibitem{CN}
%B. Coleman and W. Noll,  {\em On the Thermostatics of Continuous Media}, Arch. Rational Mech. Anal., {\bf 4} (1959), 97--128.

%\bibitem{CFG}
%D. Cordoba, D. Faraco and  F. Gancedo, {\em Lack of uniqueness for weak solutions of the incompressible porous media equation},  Arch. Ration. Mech. Anal., {\bf 200} (3) (2011), 725--746.


\bibitem{Da}
B. Dacorogna, ``Direct methods in the calculus of variations," Second edition. Applied Mathematical Sciences, 78. Springer, New York, 2008.

\bibitem{DM1}
B. Dacorogna and P. Marcellini, ``Implicit partial differential equations," Progress in Nonlinear Differential Equations and their Applications, 37. Birkh\"auser Boston, Inc., Boston, MA, 1999.

%\bibitem{DH}
%C. Dafermos and W. Hrusa, {\em Energy methods for quasilinear hyperbolic initial-boundary value problems. Applications to elastodynamics},  Arch. Rational Mech. Anal., {\bf 87} (3) (1985),  267--292.

%\bibitem{Ds}
%C. M. Dafermos, {\em The mixed initial-boundary value problem for the equations of nonlinear one dimensional viscoelasticity},  J. Differential Equations, {\bf 6} (1969),  71--86.

%\bibitem{Ds1}
%C. M. Dafermos, {\em Quasilinear hyperbolic systems with involutions},  Arch. Rational Mech. Anal., {\bf 94} (1986),  373--389.

\bibitem{Dy}
W. Day, ``The thermodynamics of simple materials with fading memory," Tracts in Natural Philosophy, 22. Springer-Verlag, New York, Heidelberg and Berlin, 1970.

%\bibitem{DS}
%C. De Lellis and L. Sz\'ekelyhidi Jr, {\em  The Euler equations as a differential inclusion,} Ann. of Math., {\bf 170} (3) (2009), 1417--1436.

%\bibitem{DS2} C.~De Lellis and L.~Sz\'ekelyhidi Jr, {\em On admissibility criteria for weak solutions of the Euler equations,} Arch. %Rational Mech. Anal., {\bf 195} (2010) 225--260.

%\bibitem{DST}
%S. Demoulini, D. Stuart and A. Tzavaras, {\em Construction of entropy solutions for one dimensional elastodynamics via time discretisation}, {Annales de l'I.H.P., Analyse Non Lin\'eaire}, {\bf 17} (6) (2000),  711--731.

%\bibitem{DST1}
%S. Demoulini, D. Stuart and A. Tzavaras, {\em A variational approximation scheme for three-dimensional elastodynamics with polyconvex energy}, Arch. Rational Mech. Anal., {\bf 157} (2001),  325--344.

%\bibitem{DST2}
%S. Demoulini, D. Stuart and A. Tzavaras, {\em Weak-Strong Uniqueness of Dissipative Measure-Valued Solutions for Polyconvex Elastodynamics}, Arch. Rational Mech. Anal., {\bf 205} (2012),  927--961.

%\bibitem{Di}
%R. DiPerna, {\em Convergence of approximate solutions to conservation laws,} Arch. Rational Mech. Anal., {\bf 82} (1983),  27--70.

\bibitem{Er}
J. L. Ericksen, {\em Equilibrium of bars,} J. Elasticity, {\bf 5} (1975),  191--201.

\bibitem{Es}
S. Esedoglu, {\em An analysis of the {P}erona-{M}alik scheme},  Comm. Pure Appl. Math., {\bf 54} (12) (2001),  1442--1487.

\bibitem{EG} 	
S. Esedoglu and J.B. Greer, {\em Upper bounds on the coarsening rate of discrete, ill-posed nonlinear diffusion equations},  Comm. Pure Appl. Math., {\bf  62} (1) (2009), 57--81.

\bibitem{GG1}
M. Ghisi and M. Gobbino, {\em A class of local classical solutions for the one-dimensional Perona-Malik equation}, Trans. Amer. Math. Soc., {\bf 361} (12) (2009),  6429--6446.

\bibitem{GG2}
M. Ghisi and M. Gobbino, {\em  An example of global classical solution for the Perona-Malik equation}, {Comm. Partial Differential Equations}, {\bf 36} (8) (2011),  1318--1352.

%\bibitem{Gr}
%M. Gromov, {\em Convex integration of differential relations},  Izv. Akad. Nauk SSSR Ser. Mat., {\bf 37} (1973), 329--343.

\bibitem{Gu}
P. Guidotti, {\em A backward-forward regularization of the Perona-Malik equation},  J. Differential Equations, {\bf 252} (4) (2012), 3226--3244.

%\bibitem{Hi}
%R. Hill, {\em On uniqueness and stability in the theory of finite elastic strain},  J. Mech. Phys. Solids, {\bf 5} (1957), 229--241.

\bibitem{Ho}
K. H\"ollig, {\em Existence of infinitely many solutions for a forward backward heat equation},  Trans. Amer. Math. Soc., {\bf 278} (1) (1983), 299--316.

\bibitem{HN}
K. H\"ollig and J. N. Nohel, {\em A diffusion equation with a nonmonotone constitutive function},   Systems of nonlinear partial differential equations (Oxford, 1982),  409�-422,
NATO Adv. Sci. Inst. Ser. C: Math. Phys. Sci., 111, Reidel, Dordrecht-Boston, Mass., 1983.

%\bibitem{Ka}
%C. Kahane, {\em A gradient estimate for solutions of the heat equation. II}, Czechoslovak Math. J., {\bf 51} (126) (2001), 39--44.

\bibitem{KK}
B. Kawohl and N. Kutev, {\em Maximum and comparison principle for one-dimensional anisotropic diffusion}, Math. Ann., {\bf 311} (1) (1998),  107--123.

%\bibitem{Ky}
%S. Kichenassamy, {\em The Perona-Malik paradox},  SIAM J. Appl. Math., {\bf 57} (5) (1997),  1328--1342.

\bibitem{KY}
S. Kim and B. Yan, {\em Radial weak solutions for the Perona-Malik equation  as a differential inclusion}, J. Differential Equations, {\bf 258} (6) (2015), 1889--1932.

\bibitem{KY1}
S. Kim and B. Yan, {\em Convex integration and infinitely many weak solutions to the Perona-Malik equation in all dimensions},
SIAM J. Math. Anal., {\bf 47}(4) (2015), 2770--2794.

%\bibitem{KY2}
%. Kim and B. Yan, {\em On Lipschitz solutions for some  forward-backward parabolic  equations}, Preprint.

%\bibitem{Ki}
%B. Kirchheim, {\em Rigidity and geometry of microstructures}, Habilitation thesis, University of Leipzig, 2003.

%\bibitem{KS} 	
%S. Klainerman and T. Sideris, {\em On almost global existence of nonrelativistic wave equations in 3D},  Comm. Pure Appl. Math., {\bf  49} (1996), 307--321.

\bibitem{LSU}
O. A. Lady\v{z}enskaja and V. A. Solonnikov and N. N. Ural'ceva, ``Linear and quasilinear equations of parabolic type. (Russian)," Translated from the Russian by S. Smith. Translations of Mathematical Monographs, Vol. 23 American Mathematical Society, Providence, R.I. 1968.

\bibitem{Ln}
G. M. Lieberman,  ``Second order parabolic differential equations,"  World Scientific Publishing Co., Inc., River Edge, NJ, 1996.

%\bibitem{Li}
%P. Lin, {\em Young measures and an application of compensated compactness to one-dimensional nonlinear elastodynamics},  Trans. Amer. Math. Soc., {\bf 329} (1992), 377--413.

%\bibitem{MP}
%S. M\"uller and M. Palombaro, {\em On a differential inclusion related to the Born-Infeld equations}, SIAM J. Math. Anal., {\bf 46} (4) (2014), 2385--2403.

%\bibitem{MSv1}
%S. M\"uller and V. \v Sver\'ak, {\em Convex integration with constraints and applications to phase transitions and partial differential equations}, J. Eur. Math. Soc. (JEMS), {\bf 1} (4) (1999), 393--422.

%\bibitem{MSv2}
%S. M\"uller and V. \v Sver\'ak, {\em Convex integration for Lipschitz mappings and counterexamples to regularity}, Ann. of Math. (2), {\bf 157} (3) (2003), 715--742.

\bibitem{MTT}
C.~Mascia, A.~Terracina and A.~Tesei, {\em Two-phase entropy solutions of a forward-backward parabolic equation}, Arch.\,Rational Mech.\,Anal., {\bf 194} (2009), 887--925.


\bibitem{MSy}
S. M\"uller and M. Sychev, {\em Optimal existence theorems for nonhomogeneous differential inclusions}, J. Funct. Anal., {\bf 181} (2) (2001), 447--475.

%\bibitem{Na}
%G. Nardi, {\em Schauder estimate for solutions of Poisson's equation with Neumann boundary condition,} arXiv:1302.4103
%(2013).

\bibitem{Pa}
V.~Padr\'on, {\em Effect of aggregation on population recovery modeled by a forward-backward pseudoparabolic equation}, Trans.\,Amer.\,Math.\,Soc., {\bf 356} (2003), 2739--2756.

%\bibitem{Pe}
%R. Pego, {\em Phase transitions in one-dimensional nonlinear viscoelasticity: admissibility and stability}, Arch. Rational Mech. Anal., {\bf 97} (1987),  353--394.

%\bibitem{PS}
%R. Pego and D. Serre, {\em Instabilities in Glimm's scheme for two systems of mixed type}, SIAM J. Numer. Anal., {\bf 25} (5) (1988),  965--988.

\bibitem{PM}
P. Perona and J. Malik, {\em Scale space and edge detection using anisotropic diffusion}, IEEE Trans. Pattern Anal. Mach. Intell., {\bf 12} (1990),  629--639.

%\bibitem{PW}
%M. Protter and H. Weinberger, ``Maximum principles in differential equations," Prentice-Hall, Inc., Englewood Cliffs, N.J. 1967.


%\bibitem{Po}
%L. Poggiolini, {\em Implicit pdes with a linear constraint}, Ricerche Mat.,  {\bf 52} (2)  (2003),   217--230.

\bibitem{Sl}
M.~Slemrod, {\em Dynamics of measure valued solutions to a backward-forward heat equation}, J.\,Dyn.\,Diff.\,Equation, {\bf 3}(1991), 1--28.


%\bibitem{Pr}
%A. Prohl, {\em Convergence of a finite element-based space-time discretization in elastodynamics}, SIAM J. Numer. Anal.,  {\bf 46} (5)  (2008),   2469-2483.

%\bibitem{Ri}
%M. Rieger, {\em Young measure solutions for nonconvex elastodynamics}, SIAM J. Math. Anal., {\bf 34} (2003),   1380--1398.

%\bibitem{Se}
%D. Serre, {\em Relaxations semi-lin\'eaire et cin\'etique des syst\`emes de lois de conservation}, {Annales de l'I.H.P., Analyse Non Lin\'eaire}, {\bf 14} (1) (1997),  143--162.

%\bibitem{Sr}
%J. Shearer, {\em Global existence and compactness in $L^p$ for the quasi-linear wave equation}, {Comm. Partial Differential Equations}, {\bf 19} (1994),  1829--1877.

%\bibitem{Sh}
%M. Shearer, {\em The Riemann Problem for a Class of Conservation Laws of Mixed Type}, {J. Differential Equations}, {\bf 46} (1982),  426--443.

\bibitem{Sy}
R. Shvydkoy, {\em Convex integration for a class of active scalar equations}, J. Amer. Math. Soc., {\bf 24} (4) (2011), 1159--1174.

%\bibitem{TTZ}
%S. Taheri, Q. Tang and Z. Zhang, {\em Young measure solutions and instability of the one-dimensional Perona-Malik equation}, J. Math. Anal. Appl., {\bf 308} (2) (2005), 467--490.

%\bibitem{Ta}
%L. Tartar, ``Compensated compactness and applications to partial differential quations, in: Knops Nonlinear Analysis and Mechanics, IV Heriot-Watt Symposium, Vol. IV," Pitman Research Notes in Mathematics, Pitman, Boston, 1979, pp. 136--192.

\bibitem{Tr}
C. Truesdell, ``Rational thermodynamics," 2nd ed., Springer-Verlag, New York, 1984.

%\bibitem{Tz}
%A. Tzavaras, {\em Materials with internal variables and relaxation to conservation laws,} Arch. Rational Mech. Anal., {\bf 146} (1999),  129--155.

\bibitem{V}
A.~Visintin, {\em Forward-backward parabolic equations and hysteresis}, Calc. Var. Partial Differential Equations  {\bf 15}  (2002),  115--132.

%\bibitem{Ya1}
%B. Yan, {\em On the equilibrium set of magnetostatic energy by differential inclusion}, Calc. Var. Partial Differential Equations  {\bf 47}  (3-4) (2013),  547--565.

%\bibitem{Ya2}
%B. Yan, {\em On stability and asymptotic behaviours for a  degenerate Landau-Lifshitz equation}, Proc. Roy. Soc. Edinburgh Sect. A, to appear.

\bibitem{Zh}
K. Zhang, {\em Existence of infinitely many solutions for the one-dimensional Perona-Malik model}, Calc. Var. Partial Differential Equations, {\bf 26} (2) (2006), 171--199.

\bibitem{Zh1}
K. Zhang, {\em On existence of weak solutions for one-dimensional forward-backward diffusion equations}, J. Differential Equations, {\bf 220} (2) (2006), 322--353.

\end{thebibliography}
\end{document}